\documentclass[12 pt, a4paper]{article}
\usepackage{graphicx,mathtools,amsmath,amssymb,amsthm,enumerate, hyperref, xcolor,mathrsfs}
\usepackage{marvosym, fontawesome}
\usepackage{url}
\usepackage{enumitem}

\usepackage[doi=false, url =false, isbn = false, style=alphabetic, sorting=nyt, backend=biber, maxalphanames=6, maxbibnames=6, uniquelist=false]{biblatex}
\newtheoremstyle{myremark}     {10pt}{10pt}{}{}{\bfseries}{.}{.5em}{}

\newtheoremstyle{mynotation}     {10pt}{10pt}{}{}{\bfseries}{.}{.5em}{}

\newtheorem{thm}{Theorem}[section]
\newtheorem{cor}[thm]{Corollary}
\newtheorem{lem}[thm]{Lemma}
\newtheorem{ques}[thm]{Question}
\newtheorem{prop}[thm]{Proposition}
\newtheorem{letterthm}{Theorem}
\newtheorem{lettercor}[letterthm]{Corollary}

\theoremstyle{definition}
\newtheorem{defn}[thm]{Definition}
\newtheorem{exmp}[thm]{Example}

\theoremstyle{myremark}
\newtheorem{rem}[thm]{Remark}

\theoremstyle{mynotation}

\usepackage{etoolbox}
\preto\theorem{\medskip}

\DeclareMathOperator{\Cob}{Cob}
\DeclareMathOperator{\fin}{fin}
 
 \DeclareMathOperator{\IM}{Im}
 
 \DeclareMathOperator{\mespres}{mp}
 \DeclareMathOperator{\Affn}{Affn}

 \newcommand{\C}{\mathbb{C}}

 \newcommand{\N}{\mathbb{N}}

 \newcommand{\R}{\mathbb{R}}

 \newcommand{\actson}{\curvearrowright}

\newcommand{\defeq}{\vcentcolon=}

\def\C{\mathbb C}

\def\F{\mathbb F}

\def\N{\mathbb N}
\def\Q{\mathbb Q}
\def\R{\mathbb R}
\def\T{\mathbb T}
\def\Z{\mathbb Z}

\newcommand{\CB}{\mathcal{B}}

\newcommand{\CF}{\mathcal{F}}

\newcommand{\CH}{\mathcal{H}}

\DeclareMathOperator{\Fix}{Fix}

\DeclareMathOperator{\id}{id}
\let\ker\relax                                        
\DeclareMathOperator{\ker}{Ker}

\DeclareMathOperator{\Stab}{Stab}
\DeclareMathOperator{\GL}{GL}                                                   
\DeclareMathOperator{\SL}{SL}

\newcommand{\Aut}{\operatorname{Aut}}

\newcommand{\Ann}{\text{Ann}}

\makeatletter
\def\thanks#1{\protected@xdef\@thanks{\@thanks
        \protect\footnotetext{#1}}}
\makeatother

\newlength\mylen
\newlist{case}{enumerate}{1}
\setlist[case,1]{label=\textbf{Case~\arabic*.}, 
  labelwidth=\dimexpr-\mylen-\labelsep\relax,leftmargin=0pt,align=right}

\title{On ergodicity of linear actions on $\R^n$ and factoriality of group von Neumann algebras}

\author{
  Soham Chakraborty\textsuperscript{1} \and
  Chinmay Tamhankar\textsuperscript{2}
}

\date{\today}  % Add a date if needed

\begin{document}

\maketitle

\begingroup
\renewcommand\thefootnote{\arabic{footnote}}
\footnotetext[1]{\hspace{0 em} \faMapMarker : Départment de Mathématiques et Applications, École Normale Supérieure, 45 Rue d'Ulm, 75005 Paris, France. \Letter: \texttt{soham.chakraborty@ens.psl.eu}}

\footnotetext[2]{\hspace{0 em} \faMapMarker : Department of Mathematics, Indian Institute of Technology Madras, 600036 Chennai, India. \Letter: \texttt{chinmay.7.tamhankar@gmail.com}}
\endgroup

\begin{abstract}
    We give some natural conditions on actions of discrete countable groups on abelian locally compact groups of Lie type that imply factoriality of the group von Neumann algebras of their semidirect products. This allows us to give a fairly large class of examples of locally compact groups whose group von Neumann algebras are factors. 
\end{abstract}
\section{Introduction}

Since the advent of the theory of von Neumann algebras, group von Neumann algebras have arguably provided the most important class of examples. Recall that a von Neumann algebra is a $\ast$-subalgebra of the bounded operators $\CB(\CH)$ on a Hilbert space $\CH$ that is closed under the strong operator topology. When the $\ast$-subalgebra is generated by the left regular representation of a locally compact second countable group $G$, the associated von Neumann algebra is called the group von Neumann algebra of $G$ and typically denoted by $L(G)$. In the class of von Neumann algebras, the ones with trivial center are called \textit{factors} and in some sense they form the building blocks of all von Neumann algebras. 

Murray and von Neumann introduced group von Neumann algebras in \cite{MurrayvN4} and showed in \cite[Lemma 5.3.4]{MurrayvN4} that for a discrete countable group $G$, the group von Neumann algebra $L(G)$ is a factor if and only if every non-trivial conjugacy class of $G$ is infinite. Such groups are typically called \textit{infinite conjugacy class (icc)} groups. Beyond the discrete case, however, determining an intrinsic characterization of locally compact groups whose von Neumann algebras are factors remains a challenging open question even today. In fact as in \cite{Vaes25}, such a property should not pass to cocycle twists, which indicates that a simple intrinsic characterization might not even exist. 

There have been some examples of factorial locally compact groups in the literature. For a nice historical overview, we refer the reader to the introduction of \cite{Morando25}. A major class of examples that arises in the literature exploits techniques from ergodic theory. 
It follows from \cite[Theorem VIII]{vonNeumannIII} (see also \cite[Proposition 2.2]{Sutherland78}, \cite[Proposition 14.D.1]{Bekka-delaharpe2020}), that for a non-singular action $\alpha$ of a discrete countable group $\Gamma$ on an abelian locally compact group $N$ by continuous automorphisms, the von Neumann algebra of the semidirect product $L(N \rtimes_{\alpha} \Gamma)$ is isomorphic to the crossed product $L^{\infty}(\widehat{N}) \rtimes_{\widehat{\alpha}} \Gamma$. Here, $\widehat{\alpha}: \Gamma \actson \widehat{N}$ refers to the dual action on the Pontryagin dual of $N$. 

In such a setting, there are precise conditions on the dual action $\Gamma \actson \widehat{N}$ such that the crossed product is a factor (see \cite{Vaes20} or 
\cite[Proposition 6.3]{BCDK24}). This idea has already been used to give examples of factors in \cite{Sutherland78} where Sutherland exploits this to provide two distinct classes (injective and non-injective) of examples of group von Neumann algebras that are factors of any given type. One important necessary condition for $L(N \rtimes \Gamma)$ to be a factor is ergodicity of the dual action $\Gamma \actson \widehat{N}$. When $N$ is countable discrete, ergodicity of $\Gamma \actson \widehat{N}$ turns out to be equivalent to the non-trivial orbits of $\Gamma \actson N$ being infinite. We attempt to achieve a similar characterization of ergodicity of the dual action when $N$ is no longer discrete. 

We deal with the case when $N$ is an abelian group of Lie type, i.e., $N$ is isomorphic to $ \R^m \times \T^n \times D$ for a discrete abelian group $D$. While we do not succeed in giving a complete characterization, we show that under some natural mild conditions on the action $\Gamma \actson N$, the group von Neumann algebra $L(N \rtimes \Gamma)$ is a factor. This allows us to give a somewhat large class of examples of factorial group von Neumann algebras. The first goal for us is to determine if an analogous ergodicity result like the discrete case (see Proposition \ref{Prop: ergodicity iff infinite orbits}) is true for measure preserving actions on $N = \R^n$. In this case one can assume that $\Gamma < \{g \in \GL(n,\R) \; | \; |\det(g)| = 1\}$ and the dual action is precisely the linear action composed with the automorphism $A \mapsto (A^T)^{-1}$. We shall denote the image of a group $\Gamma$ under this automorphism by $\Gamma^T$. In this regard, we are interested in the following natural questions: 

\begin{ques}
\label{question 1}
    For a countable subgroup $\Gamma < \SL(n,\R)$ is it true that $\Gamma \actson \R^n$ is ergodic if and only if $\Gamma^T \actson \R^n$ is ergodic?
\end{ques}

\begin{ques}
\label{question 2}
    Is there a reasonable intrinsic characterization of countable subgroups of $\SL(n,\R)$ which act ergodically on $\R^n$?
\end{ques}

The answer to Question \ref{question 1} is yes when $n = 2$ as the map $A \mapsto (A^T)^{-1}$ is inner. In higher dimensions, lattices and dense subgroups act ergodically and on the other hand abelian groups act non-ergodically (see Example \ref{example: ergodic actions on R^n}). Since these properties are preserved under taking the transpose group, it provides some evidence towards a positive answer to Question \ref{question 1} in general. Surprisingly, Question \ref{question 1} has a negative answer and we give a counterexample in Proposition \ref{Prop: main counterexample}.
Notice that this indicates a negative answer to Question \ref{question 2} as well, since such an intrinsic characterization is not preserved under the group automorphism $A \mapsto (A^T)^{-1}$ of $\SL(n,\R)$. 

\begin{letterthm}
    For all $n \geq 3$, there is a countable discrete subgroup $\Gamma < \SL(n,\R)$ such that the linear action $\Gamma \actson \R^n$ is ergodic while the dual action $\Gamma^T \actson \R^n$ is non-ergodic.
\end{letterthm}
As in Definition \ref{Def: dually ergodic}, we call an action $\Gamma \actson \R^n$ \textit{dually ergodic} if $\Gamma^T \actson \R^n$ is ergodic. If $N$ is an abelian group of Lie type (with its standard Borel structure and left invariant Haar measure), then $N \cong N^\circ \times N/N^\circ$ where $N^\circ$ is the connected component of the identity (after choosing a splitting). If $\Phi =  \Gamma \actson N$ is an action by continuous automorphisms, then by connectedness, we have that $\Phi_g(N^\circ) = N^\circ$. The action $\Phi$ is then given by a triple $(\delta,\eta, \omega)$ where $\delta: \Gamma \actson N^\circ$ and $\eta: \Gamma \actson N/N^\circ$ are actions by continuous automorphisms and $\omega: \Gamma \times N/N^\circ \rightarrow N^\circ$ is a `twisted cocycle' (see Remark \ref{Rem: form of actions}). In Proposition \ref{Prop: diagonal action ergodic implies skew product ergodic}, we show that if the dual diagonal action $=\widehat{\delta} \times \widehat{\eta}$ on the Pontryagin dual $\widehat{N}$ is ergodic, then $\widehat{\Phi}$ is ergodic, irrespective of the cocycle $\omega$. 

For ergodicity of the diagonal action $\widehat{\delta} \times \widehat{\eta}$, ergodicity of the component actions are not enough, as demonstrated in Example \ref{Exmp: individual ergodicity not enough}. Hence we have to put a mixing condition on $\Gamma \actson N^\circ$ or $\Gamma \actson N/N^\circ$. Since $N/N^\circ$ is a discrete countable group in this setting, the dual is a compact abelian group and the action $\Gamma \actson \widehat{N/N^\circ}$ is automatically probability measure preserving (pmp) with respect to the Haar measure. Hence if either the non-singular action $\Gamma \actson N^\circ$ is weakly mixing or the pmp action $\Gamma \actson \widehat{N/N^\circ}$ is mixing, we can apply an `ergodicity of diagonal action' result, as in \cite[Theorem 1.1]{GLASNER_WEISS_2016} or \cite[Theorem 2.3]{SchmidtWalters82}. For linear actions on $\R^n$, there are natural classes of examples of actions that are doubly ergodic, for example, actions of lattices (for $n \geq 3$) and dense subgroups of $\SL(n,\R)$ on $\R^n$, and consequently these actions are weakly mixing. Combining this, we prove our second main result: 

\begin{letterthm}
\label{Thm: main theorem on R^n times D, diagonal action ergodic}
    Let $\Gamma$ be a countable discrete group and $\Phi = (\delta,\eta,\omega): \Gamma \actson N$ be an action by continuous Haar-preserving automorphisms on a locally compact abelian group $N$ of Lie type. Suppose that $N$ has no non-trivial compact connected subgroups. Suppose one of the following conditions are satisfied: 
    \begin{enumerate}
        \item $\delta: \Gamma \actson N^\circ$ is faithful and dually doubly ergodic and every orbit (except at the identity) of $\eta: \Gamma \actson N/N^\circ$ is infinite.
        \item $\delta: \Gamma \actson N^\circ$ is faithful and dually ergodic and every stabilizer (except at the identity) of $\eta: \Gamma \actson N/N^\circ$ is finite.  
    \end{enumerate}
    Then $L(N \rtimes_\Phi \Gamma)$ is a factor. Conversely, if $L(N \rtimes_\Phi \Gamma)$ is a factor then $N$ has no compact connected subgroups and $\delta$ is dually ergodic.  
  
\end{letterthm}

We refer the reader to Remark \ref{rem: type of the factor} for a discussion on the type of the resulting factor. Theorem \ref{Thm: main theorem on R^n times D, diagonal action ergodic} allows us to give new examples of factorial semidirect products. In particular we show the following (see Remark \ref{Rem: form of actions} for the notation): 

\begin{lettercor}
\label{Corr: splitting actions factoriality for lattices}
    Let $n \geq 3$ and let $\Gamma < \SL(n,\R)$ be either a lattice or a countable dense subgroup and let $\delta: \Gamma \actson \R^n$ be the linear action. Let $\eta: \Gamma \actson D$ be any action on a discrete abelian group such that every non-trivial $\Gamma$-orbit in $D$ is infinite (see Example \ref{Example: splitting actions on R^m times Z^n}  when $D = \Z^m$). Consider any multiplicative $\delta$- cocycle $\omega: \Gamma \times D \rightarrow \R^n$ for $\eta$. Let $\Phi = (\delta,\eta, \omega)$ be the twisted skew product action on $\R^n \times D$ and let $G = (\R^n \times D) \rtimes_{\Phi} \Gamma$. Then $L(G)$ is a factor. In particular, when $\omega$ is trivial, the diagonal product action gives a factor.     
\end{lettercor}

In fact it is not necessary for the diagonal action to be ergodic for the ergodicity of the twisted action $\widehat{\Phi}$. For example, if the dual action is on $\R^n \times K$ for a compact abelian group $K$, then the component action on $K$ can be trivial. In this case the entire action is a skew product action with a genuine cocycle $c: \Gamma \times \R^n \rightarrow K$. Such a skew product is ergodic if and only if the essential range of the cocycle is $K$ (see \cite{SchmidtLectureNotes}). In Proposition \ref{Prop: stabilizer ergodicity} we give a generalization of Schmidt's skew product ergodicity criterion that handles twisted skew products.  This lets us prove the final main result of this article (see Remark \ref{Rem: form of actions} and Definition \ref{Def: Cob_fin} for notation).

\begin{letterthm}
        \label{Thm: diagonal action nor ergodic}
    Let $\Phi = (\delta,\eta,\omega): \Gamma \actson N \cong N^\circ \times N/N^\circ$ be a Haar-preserving action as before and suppose that $N$ has no compact connected subgroups. Let $\Gamma_\chi$ be the $\eta$-stabilizer of a point $\chi \in N/N^\circ$. Suppose that the following hold: 
    \begin{enumerate}
        \item $\delta: \Gamma \actson N^\circ$ is faithful, 
        \item For each $\chi \in N/N^\circ$, the action $\Gamma_\chi \actson N^\circ$ is dually ergodic,
        \item $\Cob_{\fin}(\omega) = \{e\}$
    \end{enumerate}
    Then $L(N \rtimes_{\Phi} \Gamma)$ is a factor. In particular when $\eta$ is trivial, then condition 2 becomes: $\delta$ is dually ergodic.
\end{letterthm}

The main novelty of this article is that we deal with the twisted cocycles and prove ergodicity results of twisted skew products. Thus the main technical efforts of this article are Propositions \ref{Prop: diagonal action ergodic implies skew product ergodic} and \ref{Prop: stabilizer ergodicity}, where we analyze invariant functions under a twisted skew product action using properties of the twisted cocycles and a strong form of ergodicity of the action on $N^\circ$. 

A natural question that arises from this article is if a similar criterion for factoriality can be found when $N$ is not of Lie type anymore, in particular when $N$ is a totally disconnected (non-discrete) abelian group. As it is often the case in the world of locally compact groups, the situation for Lie groups and totally disconnected groups probably requires different techniques. We point out here that in a recent preprint \cite{Morando25}, Morando provides a criterion for factoriality of totally disconnected group von Neumann algebras and proves for instance that Neretin groups, as well as certain HNN extensions are factorial (see also \cite{Suzuki12} and \cite{Raum19}). It would be interesting to see if some of these methods can be adapted to this setting. Most of our results have obvious counterparts for non-measure preserving actions, in which case the semidirect product groups are non-unimodular. We refer the reader to some recent factoriality results by Miyamoto in \cite{Miyamoto} for the class of almost unimodular groups, introduced by Guinto and Nelson in \cite{guinto2025unimodulargroups}.  
\medskip

\textbf{Organisation of the article:} Section \ref{Sec: preliminaries} contains the necessary preliminaries on Pontryagin duality theory of abelian groups and some techniques from ergodic theory. In Section \ref{Sec: actions on discrete abelian groups}, we give an overview of factoriality results (mostly implicit in the literature) for actions on countable groups. In Section \ref{Sec: actions on R^m}, we deal with actions on connected abelian groups, in particular for $\R^n$. In Section \ref{Sec: cocycles} we develop some new techniques to handle ergodicity of twisted skew product actions. We use this to deal with the case of actions on general abelian groups of Lie type in Section $\ref{Sec: The general case: double ergodicity}$.
\medskip

\textbf{Acknowledgements:} S.C. is supported by the ERC advanced grant 101141693 titled \textit{Noncommutative ergodic theory of higher rank lattices}. S.C. would like to thank Cyril Houdayer and Milan Donvil for some insightful comments on an earlier version of the article and Amine Marrakchi and Basile Morando for some helpful discussions. C.T. is supported by the grant SB22231267MAETWO008573 IoE Phase II. C.T. was partially supported by IIT Madras, IRL ReLaX (CNRS IRL2000), and the grant CNRS IEA GAOA for his travel and stay in Paris during April-May 2025. 

\medskip

\textbf{Mathematics Subject Classification 2020:} 37A15, 46L10, 22D15, 22D45. 

\section{Preliminaries}
\label{Sec: preliminaries}

Unless otherwise stated, in this article we only work with von Neumann algebras with separable preduals, and with locally compact second countable groups. Unless otherwise stated, $\Gamma$ will always denote an infinite countable discrete group. 

\subsection{Ergodic theory of nonsingular actions}

In this section we shall recall some notions and results from ergodic theory, mostly of non-singular actions. Recall that a measurable action of a countable discrete group $\Gamma$ on a standard $\sigma$-finite measure space $(X,\mu)$ is called \textit{nonsingular} if $g_* \mu$ and $\mu$ are absolutely continuous for all $g \in \Gamma$. It is called \textit{measure preserving} if $g_* \mu = \mu$ for all $g \in \Gamma$ and \textit{probability measure preserving (pmp)} if it is measure preserving and $\mu(X) = 1$. A nonsingular action is called ergodic if any $\Gamma$-invariant Borel subset $E \subset X$ is either null or conull. Equivalently, the von Neumann algebra of $\Gamma$-invariant functions, denoted by $L^{\infty}(X,\mu)^{\Gamma} = \C \cdot 1$. In this article, for any locally compact group $G$ with the usual Borel structure and Haar measure, we shall denote by $\Aut(G)$ the topological group of continuous group automorphisms with the compact open topology, by $\Aut_{\mespres}(G)$ the Polish group of Haar measure preserving transformations of $G$ as a measure space and by $\Aut_{\text{ns}}(G)$, the Polish group of nonsingular transformations of $G$, with respect to the usual topology of uniform convergence on Borel subsets. 

An important question in ergodic theory is to determine when, for two nonsingular actions $\Gamma \actson (X,\mu)$ and $\Gamma \actson (Y,\nu)$, the diagonal action $\Gamma \actson (X \times Y, \mu \times \nu)$ is ergodic. A non-singular action $\Gamma \actson (X,\mu)$ is called \textit{weakly mixing} if for any pmp ergodic action $\Gamma \actson (Y,\nu)$, the diagonal action $\Gamma \actson X \times Y$ is ergodic. A non-singular action $\Gamma \actson (X,\mu)$ is called \textit{doubly ergodic} if the diagonal action $\Gamma \actson X \times X$ is ergodic. A pmp action is doubly ergodic if and only if it is weakly mixing. However, for non-singular actions, double ergodicity is strictly stronger than weak mixing. We refer the reader to \cite{GLASNER_WEISS_2016} for a more elaborate discussion and for proofs of the results mentioned here.

Another important property that we are going to use in this article is the notion of a mixing action. A pmp action $\Gamma \actson (X,\mu)$ is called \textit{mixing} if for any non-null Borel subsets $E,F \subseteq X$, we have that $\mu(E \cap gF) - \mu(E)\mu(F) \rightarrow 0$ as $g \rightarrow \infty$. It can be checked that mixing implies weakly mixing which in turn implies ergodicity. Recall that an ergodic action is called \textit{properly ergodic} if every orbit has measure zero. It turns out that a pmp action $\Gamma \actson (X,\mu)$ is mixing if and only if for every properly ergodic non-singular action $\Gamma \actson (Y,\nu)$, the diagonal action $\Gamma \actson X \times Y$ is ergodic. We record this theorem below, which is due to Schmidt and Walters \cite{SchmidtWalters82}. 

\begin{thm}\textup{\cite[Theorem 2.3]{SchmidtWalters82}}
\label{Thm: Schmidt-Walters theorem}
    Let $\Gamma \actson (X,\mu)$ be pmp mixing and let $\Gamma \actson (Y,\nu)$ be non-singular and properly ergodic. Then the diagonal action $\Gamma \actson (X \times Y, \mu \times \nu)$ is ergodic.
\end{thm}

\subsection{Cocycles and skew actions}\label{Subsec: skew actions}

Let $\Gamma \actson (X,\mu)$ be a nonsingular action of a discrete group $\Gamma$ on a topological group $X$ with Haar measure $\mu$ by continuous automorphisms. A Borel \textit{1-cocycle with target $T$} is a Borel map $\omega: \Gamma \times X \rightarrow T$, where $T$ is a locally compact group, satisfying $\omega(h,gx)\omega(g,x) = \omega(hg,x)$ for a.e. $x \in X$ and all $g,h \in \Gamma$. Any group homomorphism $\pi: \Gamma \rightarrow T$ induces a cocycle $\omega_{\pi}$ given by $\omega_{\pi} (g,x) = \pi(g)$. Two cocycles $\omega$ and $\omega'$ are said to be \textit{cohomologous} if there is a Borel function $f: X \rightarrow T$ satisfying $\omega'(g,x) = f(gx)\omega(g,x)f(x)^{-1}$ for all $g \in \Gamma$ and a.e. $x \in X$. A cocycle $\omega$ which is cohomologous to the trivial cocycle 1 is called a \textit{coboundary}. 

We now recall the notion of the essential range of a cocycle. Let $\omega: \Gamma \times X \rightarrow T$ be a cocycle for an \textit{ergodic} action $\Gamma \actson X$ and let $\overline{T}$ be the one-point compactification of $T$. An element $t \in \overline{T}$ is called an \textit{essential value} of $\omega$ if for every open set $U \subset \overline{T}$ containing $t$ and every non-null Borel subset $E \subseteq X$, there is an element $g \in \Gamma$ and a non-null Borel subset $E_0 \subseteq E$ such that $gE_0 \cup E_0 \subseteq E$ and $\omega(g,E_0) \subseteq U$. The set of essential values is usually denoted by $\overline{E}(\omega)$. By \cite[Lemma 3.3]{SchmidtLectureNotes}, the set $E(\omega) \coloneqq \overline{E}(\omega) \cap T$ is a closed subgroup of $T$, and is called the \textit{essential range of $\omega$}. In this setting, the \textit{skew action} with respect to $\omega$ is the action $\Gamma \actson X \times T$ given by: 
\begin{align*}
    g \cdot (x,t) \coloneqq (gx, \omega(g,x)t)
\end{align*}

It turns out, by \cite[Corollary 5.4]{SchmidtLectureNotes} when $T$ is abelian and by \cite[Theorem 2.2]{Kim06} in general, that the skew action $\Gamma \actson X \times T$ is ergodic if and only if the essential range $E(\omega) = T$. In this situation we often say that $\omega$ has \textit{dense range} in $T$. 

Suppose now that $T$ is abelian. A continuous cocycle $\omega$ is said to be \textit{multiplicative} if we have $\omega(g,xy) = \omega(g,x)\omega(g,y)$ for all $g\in \Gamma$ and $x,y\in X$. If $\omega$ is multiplicative, then the skew action on $X\times T$ is via group automorphisms and acts trivially on $T$. Conversely, suppose that $\beta$ is a continuous action of $\Gamma$ on $X\times T$ that is trivial on $T$. Then, $\alpha, \omega$ defined by
\begin{align*}
    \beta_g(x,e) = (\alpha_g(x), \omega(g,x))\,\quad \forall g\in G,\, x\in X
\end{align*}
satisfy that $\alpha$ is a continuous action of $\Gamma$ on $X$ and $\omega$ is a continuous multiplicative $1$-cocycle for $\alpha$. Further, $\beta$ can be recovered as the skew action.

\subsection{Locally compact groups and Haar measure}

Recall that a locally compact group $G$ comes equipped with a left invariant Haar measure, i.e. a $\sigma$-finite measure $\mu$ on $G$ such that $\mu(gE) = \mu(E)$ for all Borel subsets $E \subseteq G$. The Haar measure is moreover unique up to multiplication by positive scalars. The Haar measure takes finite values on every compact subset of $G$ and is strictly positive for every open subset of $G$. In particular when $G$ is compact, $\mu(G) < \infty$. The \textit{normalized Haar measure} on a compact group $G$ is the unique left invariant Haar measure satisfying $\mu(G) = 1$. On a compact group $G$, one can show that the normalized Haar measure is also right invariant. In general a locally compact group $G$ where the left-invariant Haar measure is also right invariant, is called \textit{unimodular}. For example, discrete countable groups with the counting measure are unimodular. All abelian groups are unimodular: for example $\R^n$ with respect to addition form a topological group which is locally compact and the usual Lebesgue measure $\lambda$ is the left invariant Haar measure. Moreover $\lambda$ is also right invariant and hence $\R^n$ is unimodular. 

In general the difference between the left and the right invariant Haar measures can be quantified as follows. Given a locally compact group $G$ with a left invariant Haar measure $\mu$, there is a continuous group homomorphism $\Delta: G \rightarrow \R^+_*$ such that for all $g \in G$ and Borel subset $E \subset G$ we have: 

\begin{align*}
    \mu(Eg) = \Delta(g^{-1})\mu(E)
\end{align*}

The homomorphism $\Delta$ is usually called the \textit{modular function} on $G$. By uniqueness of the Haar measure up to positive scalars, $\Delta$ is well defined as a function on $G$. For any integrable function $f \in L^1(G)$, we have that: 
\begin{align*}
    \int_{G} f(gh) d\mu(g) = \Delta(h^{-1}) \int_{G} f(g) d\mu(g)
\end{align*}

Suppose $\Gamma$ is a discrete countable group and $N$ is a locally compact abelian group. Suppose that $\Gamma$ acts on $N$ by continuous group automorphisms. Then the semidirect product $G = N \rtimes \Gamma$ is a locally compact group. It can be checked that $G$ is unimodular if and only if the action $\Gamma \actson N$ preserves the Haar measure on $N$.

For example, consider the linear action of a countable subgroup $\Gamma < \GL(n,\R)$ on $\R^n$ and let $G$ be the semidirect product. Then for $g \in \Gamma$ and a Borel subset $E \subset \R^n$, we have that $\lambda(g \cdot E) = |\det(g)| \cdot \lambda(E)$. As a consequence, we have that $\Delta(g,x) = |\det(g)|^{-1}$ for all $g \in \Gamma$ and $x \in \R^n$. In particular, if $\Gamma$ is a subgroup of $\SL(n,\R)$, then the action preserves the Lebesgue measure and the semidirect product $G$ is unimodular.  

In the case when $K$ is a compact group and $\mu$ is the unique normalized Haar measure, any action $\Gamma \actson K$ by continuous automorphisms is automatically probability measure preserving (pmp). This is because for any $g \in \Gamma$, the pushforward $g_* \mu$ is once again a normalized left invariant Haar measure, and by uniqueness it must be equal to $\mu$. In this case the semidirect product is always unimodular. 

\subsection{Pontryagin duality for actions on abelian groups}

Let $N$ be a locally compact abelian group, then the group of continuous homomorphisms from $N$ to the torus $\T$ equipped with the topology of uniform convergence on compact subsets is called the Pontryagin dual of $N$ and is denoted by $\widehat{N}$. It can be checked that $N$ is again a locally compact abelian group. The Pontryagin duality theorem states that there is a natural isomorphism between $N$ and the dual of $\widehat{N}$. It turns out that $N$ is discrete if and only if $\widehat{N}$ is compact. For example, the integer groups $\Z^n$ and the compact torus groups $\T^n$ are Pontryagin duals of each other. The dual of $\R^n$ is $(\R^*_+)^n$, which can be identified again with $\R^n$. 

Now consider an action $\Gamma \actson N$ of a discrete countable group on a locally compact abelian group by continuous automorphisms. This induces an action $\Gamma \actson \widehat{N}$ by $(g \cdot \chi)(n) = \chi(g^{-1}n)$. This is called the \textit{dual action} of $\Gamma \actson N$ and will be essential in the rest of this article. It can be checked that $\Gamma$ preserves the Haar measure on $N$ if and only if the dual $\Gamma$-action also preserves the Haar measure on $\widehat{N}$. For example if $\Gamma \actson K$ for a compact abelian group $K$, the action is automatically Haar-measure preserving. The dual action in this case is simply an action on a discrete countable group which trivially preserves the Haar measure. For a countable subgroup $\Gamma < \SL(n,\R)$, the dual of the linear action $\Gamma \actson \widehat{\R^n}$ preserves the Haar measure on $\R^n$. 

For an action of a countable group $\Gamma$ on an abelian group $N$, the group von Neumann algebra of the semidirect product $L(N \rtimes \Gamma)$ turns out to be precisely the group measure space von Neumann algebra corresponding to the dual action. We record this result here (see \cite[Proposition 14.D.1]{Bekka-delaharpe2020}), which will be used extensively throughout this article. 

\begin{thm}
    \label{Thm: von neumann algebra of sd product is crossed product of dual action}
    Let $\Gamma$ be a discrete group acting on a locally compact abelian group $N$ by continuous automorphisms. Let $G = N \rtimes \Gamma$ denote the semidirect product. Then the von Neumann algebra $L(G)$ is unitarily equivalent to the von Neumann algebra $L^{\infty}(\widehat{N}) \rtimes \Gamma$.
\end{thm}

\subsection{Structure of locally compact abelian groups}

In this section we briefly review some structural results about locally compact abelian groups. As in the rest of this article, unless otherwise stated every locally compact group is second countable. We begin by mentioning the three main classes of examples that we shall deal with in this article. 

\begin{exmp}
\begin{enumerate}
    \item The Euclidean space $\R^n$ with respect to addition is an abelian locally compact group. The Lebesgue measure, that we shall usually denote by $\lambda$ is the invariant Haar measure up to a scalar. It is easy to see that $\R^n$ is connected and non-compact. 

    \item The free abelian group $\Z^n$ with respect to the discrete topology is a locally compact group. The Haar measure on $\Z^n$ is the usual counting measure and $\Z^n$ is trivially totally disconnected. 

    \item The torus group $\T^n = \R^n / \Z^n$ is an example of a locally compact group which is compact. By compactness, the Haar measure can be normalized to a probability measure, and is given by taking the product of the usual Lebesgue measure. Once again $\T^n$ is connected. 
\end{enumerate}
\end{exmp}

It turns out that even if these examples seem special, they are in practice fairly general in our setting. Recall that two locally compact groups $G$ and $H$ are said to be \textit{locally isomorphic} if there exist neighbourhoods $U \subset G$ and $V \subset H$ of the identity and a homeomorphism $\phi: U \rightarrow V$ such that for all $x,y \in U$, the following happens: if $xy \in U$ then $\phi(x)\phi(y) = \phi(xy)$ and if $x^{-1} \in U$ then $\phi(x)^{-1} = \phi(x^{-1})$. 

\begin{defn}
    A locally compact abelian group $G$ is said to be of \textit{Lie type} if it is locally isomorphic to $\R^n$ for some non-negative integer $n$. 
\end{defn}

Discrete abelian groups are of Lie type because they are locally isomorphic to the trivial group. The groups $\R^n$ and $\T^n$ are locally isomorphic to $\R^n$ and are hence of Lie type. A product of groups of Lie type is also of Lie type. These facts are quite straightforward and we refer the interested reader to \cite[Example 4.2.3]{DeitmarEchterhff}.

Before stating the main results from the structure theory of locally compact abelian groups, recall that a locally compact group $G$ is called \textit{compactly generated} if there is a compact subset $K$ that generates $G$. If $G$ is discrete countable, this is the same as $G$ being finitely generated. The following `structure theorem' is a collection of results that can be found in many standard texts, for example in \cite[Corollary 7.54 and Corollary 7.56]{HofmannMorris23} and \cite[Chapter 7]{DeitmarEchterhff}.

\begin{thm}
    \label{Thm: structure of abelian locally compact groups}
    Let $N$ be a locally compact abelian group. Then the following are true: 
    \begin{enumerate}
        \item For each neighbourhood $U$ of the identity, there is a compact subgroup $K \subset U$ such that $N/K \cong \R^m \times \T^n \times D$, where $m,n$ are non-negative integers and $D$ is a countable discrete abelian group. If $N$ is of Lie type, then $N \cong \R^m \times \T^n \times D$.   
        
        \item There is a unique maximal compact connected subgroup $K_0 < N$.
        
        \item If $N$ has no non-trivial compact subgroups, then $N \cong \R^m \times D$ for a discrete torsion-free abelian group $D$.
        
        \item If $N$ is compactly generated, then $N \cong \R^m \times K \times \Z^n$ for a compact abelian group $K$.  

        \item If $N$ does not have a non-trivial compact connected subgroups then $N \cong \R^n \times L$ where $L$ has a compact open subgroup. 
    \end{enumerate}
\end{thm}

The point in the previous theorem about $N$ containing no compact subgroups is quite natural in our setting due to the following lemma. 

\begin{lem}
    \label{Lemma: Compact normal subgroup implies no factoriality}
    Let $G$ be a locally compact group that has a non-trivial compact normal subgroup $K$. Then $L(G)$ is not a factor. 
\end{lem}

\begin{proof}
    Indeed, for every compact subgroup $L$ of $G$, let $\mu_L$ be the normalized Haar measure on $L$ and extend $\mu_L$ to $G$ by defining it to be $0$ outside $L$. Then, the element
        \begin{align*}
            e_L\defeq \int \lambda_G(s) d\mu_L(s) \in L(G)
        \end{align*}
        is a projection. Further, $e_L$ is central if and only if $L$ is normal. In particular, $e_K\in L(G)$ is a central projection. Further, recall that $e_K L(G)\cong L(G/K)$, which in particular means that $e_K\neq 0$. On the other hand, if $s\in K\setminus \{e\}$, then $\lambda_s e_K = e_K$, which means that $e_K\neq 1$. Hence, $L(G)$ is not a factor.
\end{proof}

As a result of Theorem \ref{Thm: structure of abelian locally compact groups} and Lemma \ref{Lemma: Compact normal subgroup implies no factoriality}, there are some restrictions on the group, that we state now. 

\begin{cor}
\label{Cor: implications of factoriality}
    Let $\Gamma$ be a discrete countable group acting on a locally compact abelian group $N$. If $L(N \rtimes \Gamma)$ is a factor then $N \cong \R^m \times L$ where $L$ is totally disconnected. In particular if $N$ is of Lie type, then $L$ is discrete, and if $N$ is compactly generated, then $L = \Z^n$.
\end{cor}

\begin{proof}
    Consider the unique maximal compact connected subgroup $K_{\text{max}} < N$ as in point 2 of Theorem \ref{Thm: structure of abelian locally compact groups}. Then $\alpha_g(K_{\text{max}}) = K_{\text{max}}$ for all $g \in \Gamma$. Therefore $K_{\text{max}}$ is a compact normal subgroup of $N \rtimes \Gamma$. Thus by Lemma \ref{Lemma: Compact normal subgroup implies no factoriality}, we have that $K_{\text{max}}$ is trivial.

    Now, we have $N\cong \R^n\times L$ for some $n\geqslant 0$ and for some $L$ such that $L$ has a compact open subgroup $K$ by Theorem \ref{Thm: structure of abelian locally compact groups} (c.f. \cite[Theorem 4.2.1]{DeitmarEchterhff}). Let $K^{\circ}$ be the connected component of identity in $K$. Then, $K^{\circ}$ is a compact and connected subgroup of $N$ and hence $K^{\circ}$ is trivial, which implies that $K$ is totally disconnected. Now, let $L^{\circ}$ be the connected component of identity in $L$. Then, $K\cap L^{\circ}$ is nonempty, open and closed in $L^{\circ}$, which implies that $K\cap L^{\circ} = L^{\circ}$, i.e. $L^{\circ}\subseteq K$. This means that $L^{\circ}$ is trivial since $K$ is totally disconnected, which in turn implies that $L$ is totally disconnected.

    Finally, when $N$ is compactly generated, then by Theorem \ref{Thm: structure of abelian locally compact groups}, we have that $N \cong \R^m \times K \times \Z^n$ and running the same argument as above we have that $L = \Z^n$. 
\end{proof}

Clearly when $N$ is compactly generated such that $L(N\rtimes \Gamma)$ is a factor, we have that $N$ is of Lie type. We shall hence deal in the rest of this paper with actions of a discrete countable group $\Gamma$ on a locally compact abelian group of the form $N \cong \R^m \times D$ where $D$ is a discrete countable abelian group.

\section{Actions on discrete abelian groups}
\label{Sec: actions on discrete abelian groups}

Suppose we have an action of a discrete countable group $\Gamma$ on a discrete (countable) abelian group $D$ and let $G = D \rtimes \Gamma$. It is well known that factoriality of $L(G)$ is equivalent to the group $G$ having the icc property, i.e. every non-trivial conjugacy class of $G$ being infinite. In this section, we give an alternate formulation of this condition in terms of the action. 

Let $\widehat{D}$ be the Pontryagin dual of $D$ and recall that $\widehat{D}$ is compact. Recall that the dual action $\Gamma \actson \widehat{D}$ is automatically probability measure preserving (pmp). By Theorem \ref{Thm: von neumann algebra of sd product is crossed product of dual action}, $L(G) = L(D \rtimes \Gamma)$ is a factor if and only if the crossed product $L^\infty(\widehat{D}) \rtimes \Gamma$ of the pmp action $\Gamma \actson \widehat{D}$ is a factor. 

In \cite[Proposition 14.D.6]{Bekka-delaharpe2020}, the authors give equivalent conditions on $\Gamma \actson D$ for the dual action $\Gamma \actson \widehat{D}$ to be free and ergodic. They remark in \cite[Remark 14.D.3]{Bekka-delaharpe2020} that even though ergodicity is necessary for factoriality, freeness is not. 
Recall the following notion from \cite[Chapter 14]{Bekka-delaharpe2020}. 

\begin{defn}
    Let $\Gamma \actson N$ be an action on a locally compact abelian group $N$. For $g \in \Gamma$, let $N_g = \{(g \cdot n)n^{-1} \; | \; n \in N\}$. Notice that since $N$ is abelian, we have that $N_g$ is a subgroup of $N$. We also define the \textit{annihilator of $N_g$} as the subgroup $N_g^{\perp} = \{\chi \in \widehat{N} \; | \; \chi(n) = 1 \text{ for all } n \in N_g\}$ of $\widehat{N}$.   
\end{defn}

The annihilator of $N_g$ turns out to be precisely the set of fixed points of $g$ in the dual action. Indeed as in \cite[Proposition 14.D.6]{Bekka-delaharpe2020} we have: 

\begin{prop}\textup{\cite[Proposition 14.D.6]{Bekka-delaharpe2020}}
    \label{Prop: conjugacy class condition discrete case}
        Let $\Gamma \actson N$ be an action on a locally compact abelian group, and consider the dual action $\Gamma \actson \widehat{N}$. Then for a non-trivial element $g \in \Gamma$, the set of fixed points $\Fix_{\widehat{N}}(g)$  is equal to $N_{g}^{\perp}$. If furthermore $N$ is discrete, then $N_{g}^{\perp}$ is non-null if and only if $N_{g}$ is finite.
\end{prop}

Ergodicity of the dual action also has an interesting description in terms of orbits of the original action. The following result is well known and also appears in \cite[Proposition 14.D.6]{Bekka-delaharpe2020}.

\begin{prop}
\label{Prop: ergodicity iff infinite orbits}
    Let $\Gamma \actson D$ be an action on a discrete abelian group $D$ and consider the dual action $\Gamma \actson \widehat{D}$. Then the following are equivalent: 
    \begin{enumerate}
        \item $\Gamma \actson \widehat{D}$ is ergodic 
        \item The $\Gamma$-orbit $\{hn \;|\; h \in \Gamma\}$ of every non-trivial $n \in D$ is infinite.
    \end{enumerate}    
\end{prop}

A set of necessary and sufficient conditions for factoriality of the crossed product of a pmp action (more generally a non-singular action) of a discrete countable group was first recorded in \cite{Vaes20}. This follows as a corollary from the more general characterization of factoriality of groupoid von Neumann algebras (see \cite[Corollary 6.4]{BCDK24}). In our setting this translates to the following result:  

\begin{thm}
\label{Thm: discrete abelian groups crossed products}
    Let $\Gamma \actson D$ be an action of a countable infinite group on a discrete abelian group. Let $G = D \rtimes \Gamma$ denote the semidirect product. Then the following are equivalent: 
    \begin{enumerate}
        \item The action $\Gamma \actson \widehat{D}$ is ergodic and for every $g\neq e$, whenever $D_{g}^{\perp}$ is a non Haar-null, then $g$ has an infinite conjugacy class in $\Gamma$.
        
        \item For all $n \neq e$, the orbits $\Gamma \cdot n$ are infinite and for every $g\neq e$, if $D_{g}$ is finite, then $g$ has an infinite conjugacy class in $\Gamma$. 
        
        \item Every non-trivial conjugacy class of $G$ is infinite, i.e., $G$ is icc.   
        \item $L(G)$ is a II$_1$ factor. 
    \end{enumerate}
\end{thm}

\begin{proof}
     The implications ($3 \iff 4$) are well known, for example one can find a proof in \cite[Proposition 1.3.9]{AnaPopa10}. The implications ($1 \iff 2$) follow from Proposition \ref{Prop: conjugacy class condition discrete case} and Proposition \ref{Prop: ergodicity iff infinite orbits}.
     
     Recall that $(n,h)\cdot (m,g) = (n(h\cdot m),hg)$ and $(n,h)^{-1} = (h^{-1}\cdot n^{-1}, h^{-1})$. Similarly for elements $(n,h)$ and $(m,g)$ the conjugates are defined by: 
     \begin{align*}
         (m,g)(n,h)(m,g)^{-1} = (m(g\cdot n)(ghg^{-1} \cdot m^{-1}),ghg^{-1})
     \end{align*}
     
     $(2 \implies 3)$ Notice that if $n \neq e$, then $(e,g)(n,h)(e,g)^{-1} = (g \cdot n, ghg^{-1})$ and since the $\Gamma$ orbit of $n$ is infinite, this is an infinite set of elements as we vary $g \in \Gamma$, and the conjugacy class of $(n,h)$ is infinite. For an element of the form $(e,h)$ we have that $(m,g)(e,h)(m,g)^{-1}$ is equal to $(m(ghg^{-1} \cdot m^{-1}),ghg^{-1})$. Notice that by taking $g = e$, we already have that if $D_{h}$ is infinite then the conjugacy class of $(e,h)$ is infinite. If not, then by $(2)$ we have that the conjugacy class of $h$ in $\Gamma$ is infinite and hence the conjugacy class of $(e,h)$ in $G$ is also infinite. 

     $(3 \implies 2)$ The conjugacy class of an element of the form $(n,e)$ for $n \neq e$ is infinite. Hence as $D$ is abelian we have that the set $\{g\cdot n \; | \; g \in \Gamma\}$ is infinite, thus giving that the $\Gamma$-orbit of $n$ is infinite. Now for an element of the form $(e,h)$, we have that the conjugacy class in $G$ is the set $\{(n(ghg^{-1} \cdot n^{-1}),ghg^{-1}) \; |\; g \in \Gamma, n \in D\}$. If $h$ has a finite conjugacy class in $\Gamma$, then for some element $ghg^{-1}$, the set $\{n(ghg^{-1} \cdot n^{-1}) \; | \; n \in D\}$ must be infinite, i.e. $D_{ghg^{-1}}$ must be infinite. A simple calculation shows that $D_{ghg^{-1}} = g\cdot D_h$, which means that $D_h$ must be infinite.
\end{proof}

\begin{cor}
    Let $\Gamma \leq \GL(n,\Z)$ and consider the linear action $\Gamma \actson \Z^n$ and let $G = \Z^n \rtimes \Gamma$ be the semidirect product. Suppose there exists a non-trivial element $g \in \Gamma$ such that $g$ has no roots of unity as eigenvalues. Then the dual action $\Gamma \actson \T^n$ is essentially free and ergodic and $L(G)$ is a factor.     
\end{cor}
\begin{proof}
    Let $N = \Z^n$. For $h\neq I$ in $\Gamma$, notice that $N_h = \IM(h - I)$ and since $h - I \neq 0$, the image is an infinite subgroup of $\Z^n$, implying that $\Gamma \actson \widehat{N}$ is essentially free by Proposition \ref{Prop: conjugacy class condition discrete case}. The action of the infinite subgroup $ \langle g \rangle \actson \Z^n$ is free as $g$ has no eigenvalues which are roots of unity. Then every non-trivial orbit of $ \langle g \rangle \actson \Z^n$ is infinite, and hence every non-trivial orbit of $\Gamma \actson \Z^n$ is infinite. Thus $\Gamma \actson \widehat{N}$ is ergodic by Proposition \ref{Prop: ergodicity iff infinite orbits}. The result now follows from Theorem \ref{Thm: discrete abelian groups crossed products}. 
\end{proof}

For the purposes of Section \ref{Sec: The general case: double ergodicity}, we need a characterization of actions $\Gamma \actson D$ such that the dual action $\Gamma \actson \widehat{D}$ is not only ergodic, but mixing. The main result that we will use is \cite[Theorem 1.6]{Schmidt} that we state here. 

\begin{thm} \textup{\cite[Theorem 1.6]{Schmidt}}
    \label{Theorem: mixing iff no ergodic subgroups}
    Let $\Gamma \actson D$ be an infinite discrete group acting on a discrete abelian group by group automorphisms. Then the following are equivalent: 
    \begin{enumerate}
        \item The dual action $\Gamma \actson \widehat{D}$ is mixing, 
        \item For any infinite subgroup $\Lambda \leq \Gamma$, the dual action $\Lambda \actson \widehat{D}$ is ergodic,
        \item For any infinite subgroup $\Lambda \leq \Gamma$, all non-trivial $\Lambda$-orbits in $D$ are infinite,
        \item  Every non-trivial stabilizer of $\Gamma \actson D$ is finite.  
    \end{enumerate}
\end{thm}

As a trivial application of Theorem \ref{Theorem: mixing iff no ergodic subgroups}, we have that: 

\begin{cor}
\label{Corr: free action implies dual is mixing}
    If $\Gamma \actson D \setminus \{0\}$ is free for an action on a discrete abelian group $D$, then the dual action $\Gamma \actson \widehat{D}$ is mixing. When $\Gamma \leq \SL(n,\Z)$ and $\Gamma \actson \Z^n$ is the linear action: if no non-trivial element $g \in \Gamma$ has 1 as an eigenvalue, then the dual action $\Gamma \actson \T^n$ is mixing. 
\end{cor}

For subgroups of $\SL(2,\Z)$ we have the converse of Corollary \ref{Corr: free action implies dual is mixing} as well:

\begin{prop}
\label{prop: 1 not eigenvalue implies mixing}
    Consider the linear action $\Gamma \actson \Z^2$ for $\Gamma \leq \SL(2,\Z)$. The dual action $\Gamma \actson \T^2$ is mixing if and only if no non-trivial $g \in \Gamma$ has 1 as an eigenvalue.  
\end{prop}

\begin{proof}
    By Theorem \ref{Theorem: mixing iff no ergodic subgroups}, we need to show that every non-trivial stabilizer is finite. For a vector $(m,n) \in \Z^2$, let $\Lambda$ be its stabilizer in $\SL(2,\Z)$. One can check that there exists $g \in \GL(2,\Z)$ such that $g\Lambda g^{-1} = U$, where: 
    \begin{align*}
        U = \begin{pmatrix}
            1 & \Z \\ 0 & 1 
        \end{pmatrix}
    \end{align*}
    Now the stabilizer of $(m,n)$ under the $\Gamma$ action is $\Lambda \cap \Gamma$. Since $g(\Lambda \cap \Gamma)g^{-1}$ is a subgroup of $U$, $\Gamma \cap \Lambda$ is finite if and only if it is trivial. Hence the stabilizers are all finite if and only if the action is free, as required.
 \end{proof}

The following is an interesting observation about the mixing property for automorphisms on $\T^n$. 

\begin{prop}
    Let $g \in \SL(n,\Z)$ be an infinite order element. Then the dual $\Z$-action $ \langle g \rangle \actson \T^n$ is ergodic if and only if it is mixing.      
\end{prop}

\begin{proof}
    Suppose that the action is not mixing, then by Theorem \ref{Theorem: mixing iff no ergodic subgroups}, we have that there is a vector $d \in \Z^n$ such that $\Stab(d)$ is infinite. Thus $\Stab(d)$ is generated by an element of the form $g^i$ for some non zero integer $i$, which implies that the image of the point $d$ is finite. Hence by Proposition \ref{Prop: ergodicity iff infinite orbits}, this contradicts ergodicity, as required.  
\end{proof}

We end this section with some examples of actions by automorphisms of discrete groups that are free, and hence the dual actions are mixing. 

\begin{prop}
\label{Prop: examples of mixing actions by Bernoulli shifts}
    Let $\Gamma$ be a torsion-free group, $\Lambda$ be any non-trivial abelian group. Let $H = \bigoplus_{\Gamma} \Lambda$ and consider the action $\Gamma \actson H$ by Bernoulli shifts. Then the action $\Gamma \actson H \setminus \{\id\}$ is free, and hence the dual action $\Gamma \actson \widehat{H}$ is mixing.    
\end{prop}
\begin{proof}
    For $\Gamma \actson H$, let $f: \Gamma \rightarrow \Lambda$ be a non-empty finitely supported function and suppose that $g\cdot f = f$ for a non-trivial $g \in \Gamma$. Then there exists an element $k \in \Lambda$ such that $F = f^{-1}(\{k\})$ is a non-trivial finite subset in $\Gamma$. This means that $g \cdot F = F$ under the translation action. But if this is the case then $g^n \cdot F = F$ for all $n \in \N$. Since $F$ is finite, this implies that $g^k = e$ for some positive integer $k$, contradicting the fact that $\Gamma$ is torsion-free.
\end{proof}

\section{Ergodicity of linear actions on $\R^n$}
\label{Sec: actions on R^m}

Notice that by virtue of Corollary \ref{Cor: implications of factoriality}, factoriality of semidirect product group von Neumann algebras for actions on connected abelian groups, essentially boils down to characterizing factoriality for such actions on $\R^n$.

\begin{lem}
    \label{Lemma: every action on R^n is linear}
    Let $\Gamma$ be a discrete countable group and $\alpha$ be a faithful action of $\Gamma$ on $\mathbb{R}^{n}$ by continuous automorphisms. Then the action is linear, i.e. $\alpha$ is conjugate to an action $\Gamma \leq \GL(n,\mathbb{R}) \actson \mathbb{R}^{n}$. Moreover, the action is Lebesgue measure preserving if and only if $|\det(g)| = 1$ for all $g \in \Gamma$.  
\end{lem}

\begin{proof}
    Let $\alpha\in \Aut(\R^n)$, then clearly $\alpha$ is $\Z$-linear. For every $n\in \N$ and $x\in \R^n$,
    \begin{align*}
        \alpha\left(\frac{1}{n} x\right)
        = \frac{1}{n} n \alpha\left(\frac{1}{n} x\right)
        = \frac{1}{n} \alpha(x),
    \end{align*}
    which implies that $\alpha$ is in fact $\Q$-linear. The continuity of $\alpha$ then implies that $\alpha$ is $\R$-linear. Since $\alpha$ is invertible, there exists $A\in \GL(n,\R)$ such that $\alpha=A$, i.e. $\alpha x = Ax$ for every $x\in \R^n$. Further, if $\mu$ is the Haar measure on $\R^n$, then
    \begin{align*}
        \alpha_*\mu = \frac{1}{|\det(A)|} \mu,
    \end{align*}
    which proves the second part of the lemma.
\end{proof}

For what follows we identify the Pontryagin dual $\widehat{\R^n}$ with $\R^n$ via the following isomorphism.
\begin{align*}
    \R^n\ni t\mapsto \chi_t\in \widehat{\R^n},\quad \text{ where }\quad \chi_t(s) = e^{i\langle s,t \rangle}
\end{align*}
We shall henceforth not always explicitly distinguish them. Let $\phi$ ($\phi_n$ if $n$ needs to be emphasized) be the order-$2$ automorphism of $\GL(n,\R)$ given by $\phi(A) = (A^T)^{-1}$. Note that $\phi$, in particular, leaves $\SL(n,\R)$ invariant. Let $\alpha: \GL(n,\R) \actson \R^n$ be the linear action. Then for $g \in \GL(n,\R)$ and $s,t \in \R^n$, the dual action is given by:  
    \begin{align*}
        (\widehat{\alpha}_g (t))(s)
        = t(\alpha_g^{-1} (s))
        = t(g^{-1}s)
        = e^{i\langle g^{-1}s,t \rangle} 
        = e^{i\left\langle s, (g^{-1})^t t \right\rangle}
        = ((\alpha\circ \phi)_g(t)) (s)
    \end{align*}

For $\Gamma<\GL(n,\R)$, let us denote the image $\phi(\Gamma)$, consisting of the transposes of every element of $\Gamma$ by $\Gamma^T$. By the discussion above, the dual action of $\Gamma \actson \R^n$ is precisely the linear action $\Gamma^T \actson \R^n$. Let us state some examples and non-examples of measure-preserving ergodic actions on $\R^n$ by continuous automorphisms. 

\begin{exmp}
    \label{example: ergodic actions on R^n}
    \begin{enumerate}
        \item It is well known that if $\Gamma < \SL(n,\R)$ is a lattice for $n \geq 2$, then as an application of Moore's ergodicity theorem, the linear action $\Gamma \actson \R^n$ is ergodic. For a proof we refer the reader to \cite[Example 2.2.9]{Zimmer84} and \cite[Lemma 5.6]{PopaVaes11}. Since the transpose of a lattice is again a lattice, $\Gamma^T \actson \R^n$ is also ergodic. 
        
        \item If $\Gamma < \SL(n,\R)$ is a countable dense subgroup, then $\Gamma \actson \R^n$ is ergodic. Since $\SL(n,\R)$ has subgroups acting ergodically on $\R^n$, of course the action $\SL(n,\R) \actson \R^n$ is ergodic. In general for a continuous nonsingular ergodic action of a locally compact second countable group on a standard $\sigma$-finite measure space $G \actson (X,\mu)$, the action restricted to dense subgroups is ergodic. This essentially follows because the canonical group homomorphism $G \rightarrow \Aut(L^{\infty}(X,\mu))$ is continuous. Since ergodicity does not depend on the topology of $\Gamma$, we can consider $\Gamma$ as a discrete group and the action remains ergodic. Again, because the transpose $\Gamma^T$ of such a dense subgroup $\Gamma$ is dense in $\SL(n,\R)$, we have that $\Gamma^T \actson \R^n$ is ergodic.  

        \item Contrary to the above examples, measure preserving actions of abelian groups on $\R^n$ are never ergodic. A sketch of the proof for $\Z$-actions appears in a lecture note of Halmos (\cite{Halmos56}). Halmos also conjectured in \cite{Halmos56} that if a locally compact group $G$ has an automorphism that is ergodic with respect to a left-invariant Haar measure, then $G$ is compact. This conjecture was solved for $\Z$-actions in \cite{Rajagopalan66}. In general it follows from \cite[Theorem 1.1]{Dani01} that the action of any abelian subgroup of $\SL(n,\R)$ on $\R^n$ is not ergodic. Since the transpose of an abelian group is abelian, the dual action of an abelian group is also non-ergodic. 
    \end{enumerate}
\end{exmp}

Because of these special cases, it is tempting to think that an analog of Proposition \ref{Prop: ergodicity iff infinite orbits} holds for actions on $\R^n$. Indeed, the natural question is: for a countable group $\Gamma < \SL(n,\R)$, is it true that $\Gamma \actson \R^n$ is ergodic if and only if $\Gamma^T \actson \R^n$ is ergodic? Notice that for $n=2$, this is true. Indeed the automorphism $\phi: \SL(2,\R) \rightarrow \SL(2,\R)$ given by $\phi(A) = (A^T)^{-1}$ is inner, and hence $\Gamma \actson \R^2$ and $\Gamma^T \actson \R^2$ are orbit equivalent (in fact conjugate). Surprisingly, this is not true for higher dimensions, as we demonstrate here.

\begin{defn}
Let $n\geq 3$ and $e_1,\dots e_n$ be the standard basis for $\R^n$. Let $\Gamma_0 < \SL(n,\R)$ be the following discrete countable subgroup: 
    \begin{align*}
        \Gamma_0 = \{ A\in \SL(n,\Z)\, |\, Ae_1 = e_1 \}.
    \end{align*}
\end{defn}

\begin{lem}
    Let $\Lambda \actson H$ be an action of a discrete countable group on a locally compact abelian group. Suppose that there is a common fixed point $h \neq e \in H$ for all of $\Lambda$. Then the dual action $\Lambda \actson \widehat{H}$ is not ergodic.
\end{lem}
\begin{proof}
    By Pontryagin duality, the dual of
    %$\Lambda \actson \widehat{H}$
    $\widehat{H}$ is identified with
    %$\Lambda \actson H$.
    $H$. Hence one can think of $h$ as a character $h: \widehat{H} \rightarrow \T$ fixed by the dual action $\Lambda \actson \widehat{H}$. In particular $h \in L^\infty(\widehat{H})^{\Lambda}$ and clearly $h$ is not a scalar. Thus $\Lambda \actson \widehat{H}$ is not ergodic. 
\end{proof}

From the above lemma it is clear that $\Gamma_0^T \actson \R^n$ is not ergodic. However we have the following: 

\begin{prop}
\label{Prop: main counterexample}
    Let $n\geq 3$. Then the linear action $\Gamma_0 \actson \R^n$ is ergodic. 
\end{prop}

\begin{proof}
    Note that $\R^n$ can be decomposed as $\R\times \R^{n-1}$ so that $\Gamma_0$ acts trivially on the first copy of $\R$. Define $\alpha$ and $c$ by $g\cdot(0,v) = (c(g,v),\alpha_g(v))$ for all $g\in \Gamma_0$ and $v\in \R^{n-1}$. Then $\alpha$ is an action of $\Gamma_0$ on $\R^{n-1}$ and $c$ is a continuous multiplicative $1$-cocycle for $\alpha$. Further, the linear action of $\Gamma_0$ on $\R^n$ is the resulting skew action (see the final paragraph of \S \ref{Subsec: skew actions}). But, for every $A\in \SL(n-1,\Z)$, $w\in \Z^{n-1}$, and $g = \left(
    \begin{array}{cc}
       1  & w^T \\
       0  & A
    \end{array}
    \right)$, we have
    \begin{align*}
        \alpha_g(v) = Av
        \text{ and }
        c(g,v) = w^Tv
        \,\quad
        \forall\, v\in \R^{n-1}. 
    \end{align*}
    In particular, $\alpha$ is ergodic since $n-1\geq 2$. Thus, by \cite[Corollary 5.4]{SchmidtLectureNotes}, we only need to show that $E(c) = \R$.

    Since $n\geq 3$, the measurable set
    \begin{align*}
        P\defeq \{(v_1,v_2)\in \R^2\, |\, v_1\neq 0 \text{ and } v_2\notin \Q v_1 \}\times \R^{n-3}
    \end{align*}
    is well defined and co-null in $\R^{n-1}$. Now, let $U\subseteq \R$ be nonempty and open. For $w\in \Z^{n-1}$, let $P_w\defeq \{v\in P\, |\, w^Tv\in U \}$. Then, for every $v\in P$, there exist $m_1, m_2\in \Z$ such that $m_1v_1+m_2v_2\in U$, where $v_1, v_2$ are the first two coordinates of $v$. Hence, $P = \bigcup_{w\in \Z^{n-1}} P_w$.

    Now, let $F\subseteq \R^{n-1}$ be Borel non-null. Since $P$ is co-null in $\R^{n-1}$, there exists $w\in \Z^{n-1}$ such that $F_w\defeq F\cap P_w$ is non-null. Then, for $g\defeq \left(
    \begin{array}{cc}
       1  & w^T \\
       0  & I_{n-1}
    \end{array}
    \right)$, we have $g\cdot F_w = F_w$ and $c(g,F_w)\subseteq c(P_w)\subseteq U$. Hence, $E(c) = \R$.
\end{proof}

Thus we don't yet have a way to go back and forth between ergodicity of the dual action and a reasonable property of the original action. Nevertheless we can characterize factoriality of such group von Neumann algebras as follows:

\begin{thm}
\label{Thm: characterization of factoriality for actions on R^n}
    Let $\Gamma < \GL(n,\R)$ be a countable discrete group and let $\Gamma \actson \R^n$ be the linear action. Then $L(\R^n \rtimes \Gamma)$ is a factor if and only if $\Gamma^T \actson \R^n$ is ergodic. 
\end{thm}

\begin{proof}
   By Theorem \ref{Thm: von neumann algebra of sd product is crossed product of dual action}, we have that $L(\R^n \rtimes \Gamma)$ is isomorphic to the crossed product $L^{\infty}(\R^n) \rtimes \Gamma^T$. Notice that for $\Gamma^T \actson \R^n$, the set of fixed points $\Fix(g)$ for $g \in \Gamma^T$ is the 1-eigenspace of $g$. By faithfulness of the linear action, $\Fix(g)$ is a vector space of dimension at most $n-1$ and hence $\widehat{\alpha}$ is essentially free. Thus $L(\R^n \rtimes \Gamma)$ is a factor if and only if $\Gamma^T \actson \R^n$ is ergodic as required.   
\end{proof} 

For the purposes of applications in the next section, we make the following definition: 

\begin{defn}
\label{Def: dually ergodic}
    Let $\Gamma$ be a discrete countable group acting on an abelian connected locally compact group $N$. Suppose that $N$ has no compact subgroups. Then $N \cong \R^n$ by Theorem \ref{Thm: structure of abelian locally compact groups} and $\Gamma \actson N$ is conjugate to a linear action by Lemma \ref{Lemma: every action on R^n is linear}. We shall say that $\Gamma \actson N$ is \textit{dually ergodic} or  \textit{dually doubly ergodic} if the dual linear action $\Gamma^T \actson \R^n$ is ergodic or doubly ergodic respectively.  
\end{defn}

Now using Theorem \ref{Thm: structure of abelian locally compact groups} and Corollary \ref{Cor: implications of factoriality}, we can write Theorem \ref{Thm: characterization of factoriality for actions on R^n} in the following way: 

\begin{cor}
\label{Cor: actions on connected abelian groups}
    Let $\Gamma \actson N$ be a faithful action of a discrete countable group by continuous automorphisms on a connected locally compact abelian group. Then $L(N \rtimes \Gamma)$ is a factor if and only if $N$ has no non-trivial compact subgroups and $\Gamma \actson N$ is dually ergodic.
\end{cor}

\section{Ergodicity of twisted skew product actions}
\label{Sec: cocycles}

Let $(X,\mu)$ be a $\sigma$-finite standard measure space, $K$ be a compact abelian group with Haar measure $\nu$ and $\Gamma$ be a countable discrete group. Let $\alpha: \Gamma \actson (X,\mu)$ be a measure preserving action and let $\beta: \Gamma \actson (K,\nu)$ be an action by continuous group automorphisms. We call a Borel map $c: \Gamma \times X \rightarrow K$ a $\beta$-cocycle for $\alpha$ if it satisfies: 
\begin{align*}
    c(h,\alpha_g(x)) \beta_h( c(g,x)) = c(hg,x) \text{ for all } g \in \Gamma \text{ and a.e. } x \in X
\end{align*}
Notice that if $\beta$ is trivial then this gives a genuine 1-cocycle for $\alpha$. Such a twisted cocycle gives a twisted skew product action: 
\begin{align*}
    \rho: \Gamma \actson (X \times K, \mu \times \nu) \text{ by } \rho_g(x,k) = (\alpha_g(x), c(g,x)\beta_g(k))
\end{align*}
Once again, if $\beta$ is trivial, this gives the usual skew-product action. In the results of this section, we analyze the ergodicity of such twisted skew product actions: 

\begin{prop}
\label{Prop: diagonal action ergodic implies skew product ergodic}
     Let $\rho = \rho(\alpha,\beta,c): \Gamma \actson X \times K$ be a twisted skew product action as above and let $\rho_0 = (\alpha \times \beta)$ be the diagonal product action. If $\rho_0$ is ergodic then $\rho$ is ergodic (irrespective of the $\beta$-cocycle $c$). 
\end{prop}
\begin{proof}
     Let $F \in L^\infty(X \times K)$ be a $\rho$-invariant function. Let $\nu$ be the Haar measure on $K$ and consider the following bounded measurable function: 
    \begin{align*}
        C_F(x,t) = \int_{K} F(x,kt)\overline{F(x,k)} \; d\nu(k)
    \end{align*}
    Notice that since $\beta$ preserves $\nu$ and translation is also Haar-preserving, for any bounded measurable function $\phi \in L^\infty(K)$, and a fixed $g \in \Gamma$ and $x \in X$ we have: 
    \begin{align*}
        \int_K \phi(k) d\nu(k) = \int_K \phi(c(g,x)  \beta_g(k)) \; d\nu(k)
    \end{align*}
    This allows us to make the necessary change of variables in the following calculation. 
    \begin{align*}
        C_F(\alpha_g(x),\beta_g(t)) &= \int_K F(\alpha_g(x), k \cdot \beta_g(t)) \overline{F(\alpha_g(x),k)} d\nu(k) \\ &= \int_{K} F(\alpha_g(x), c(g,x)\beta_g(kt)) \overline{F(\alpha_g(x), c(g,x)\beta_g(k))} \; d\nu(k) 
        \\ &= \int_K F(\rho_g(x,kt))\overline{F(\rho_g(x,k))} \; d\nu(k)
        \\ &= \int_K F(x,kt)\overline{F(x,k)} \; d\nu(k)
        \\&= C_F(x,t)
    \end{align*}
    Thus $C_F$ is constant by ergodicity of $\rho_0$. Notice that for a function $\CF \in L^\infty(X \times K)$, the corresponding function $t \mapsto \CF(x,t)$ is in $L^2(K)$. Now for each character $\chi \in \widehat{K}$, and  $\CF \in L^\infty(X \times K)$, let us denote the Fourier coefficient of $t \mapsto \CF(x,t)$ at $\chi$ by $\CF_\chi = \int_K \CF(x,k)\overline{\chi(k)} \; d\nu(k)$. Now we calculate: 
    \begin{align*}
        (C_F)_\chi(x) &= \int_{K} C_F(x,k)\overline{\chi(k)} \; d\nu(k) = \int_K\int_K F(x,kt)\overline{F(x,k)\chi(t)} \; d\nu(k) d\nu(t) \\ &= \int_K\int_K F(x,u)\overline{F(x,k)\chi(uk^{-1})} \; d\nu(u) d\nu(k) \text{ by putting } u = kt
        \\&= \int_K\int_K F(x,u)\overline{F(x,k)\chi(u)}\chi(k) \; d\nu(u) d\nu(k) \\ &= \int_K F(x,u)\overline{\chi(u)} \; d\nu \int_K \overline{F(x,k)}\chi(k) \; d\nu(k) \\ &= F_\chi(x)\overline{F_\chi(x)} = | F_\chi(x) |^2
    \end{align*}

    Since $C_F$ is constant almost everywhere, let $C_F = c$. If $\chi \neq 1$, 
    \begin{align*}
        (C_F)_\chi(x) = \int_{K} c\overline{\chi(k)} \; d\nu(k) =c \cdot \int_{K} \overline{\chi(k)} \; d\nu(k) = 0
    \end{align*}
    The last equality above is because the integral of a character with respect to the Haar measure is always zero. This implies that $\|F_\chi(x)\|^2 = 0$ a.e. Thus the map $k \mapsto F(x,k)$ has all non-trivial Fourier coefficients equal to zero. Hence $F(x,k) = f_0(x)$ for a function $f_0 \in L^\infty(X)$. One checks that since $F$ is $\rho$-invariant, $f_0$ is $\alpha$-invariant. Since $\alpha$ is ergodic, this implies that $f_0$ is constant a.e. Thus $F$ is constant almost everywhere as required. 
\end{proof}

We now address the case when the diagonal product action $\rho_0 = \alpha \times \beta$ is not necessarily ergodic. In this case we can only give some partial conditions. For each $\chi \in \widehat{K}$, let $\Gamma_\chi$ be the stabilizer of the point $\chi$ for the dual action $\widehat{\beta}: \Gamma \actson \widehat{K}$. For each $\chi \in \widehat{K}$, we get an \textit{induced cocycle} $c_\chi: \Gamma_\chi \times X \rightarrow \T$ given by $c_\chi(g,x) = \chi(c(g,x))$. This is a genuine cocycle because every element of $\Gamma_\chi$ fixes $\chi$ under the dual action, indeed: 
\begin{align*}
    c_\chi(h,gx)c_\chi(g,x) &= \chi(c(h,gx))\chi(c(g,x)) = \chi(c(h,gx))\chi(\beta_h(c(g,x))) \\ &= \chi(c(h,gx)\beta_h(c(g,x))) = \chi(c(hg,x)) = c_\chi(hg,x)
\end{align*}

\begin{prop}
\label{Prop: stabilizer ergodicity}
   Suppose that $\Gamma_\chi$
 acts ergodically on $(X,\mu)$ for each $\chi \in \widehat{K}$. Then the following are equivalent: 
 \begin{enumerate}
     \item The twisted skew product $\rho: \Gamma \actson X \times K$ is ergodic
     \item For each non-trivial $\chi \in \widehat{K}$ with a finite $\widehat{\beta}$-orbit, the $\T$-valued cocycle $c_\chi$ is not a coboundary. 
 \end{enumerate}
\end{prop}

\begin{proof}
    First let us prove $2 \implies 1$. Let $F \in L^\infty(X \times K)$ be $\rho$-invariant. For $\chi \in \widehat{K}$, let us define the Fourier coefficient $F_\chi(x)$ of the function $F_x \in L^2(K)$ given by $ k \mapsto F(x,k)$ at $\chi$ as before: 
    \begin{align*}
        F_\chi(x) = \int_{K} F(x,k)\overline{\chi(k)} \; d\mu(k)
    \end{align*}
    Since $\widehat{K}$ is countable and characters form an orthonormal basis of $L^2(K)$ we have from Parseval's identity: 
    \begin{align*}
       \sum_{\chi \in \widehat{K}} |F_\chi(x)|^2 = \sum_{\chi \in \widehat{K}} |\widehat{F_x}(\chi)|^2 = \|F_x\|^2_{L^2(K)} = \int_{K} |F(x,k)|^2 d\mu(k) \leq \|F\|_\infty^2 
    \end{align*}
    Now using the $\rho$-invariance of $F$ we get for any $g \in \Gamma$:
    \begin{align*}
        F_\chi(x) &= \int_K F(x,k) \overline{\chi(k)}\; d\mu(k) = \int_K F(\alpha_g(x), c(g,x) \beta_g(k)) \overline{\chi(k)}\; d\mu(k) \\ &=  \int_{K} F(\alpha_g(x), k)\overline{\chi(\beta_{g^{-1}}(c(g,x)^{-1}k))}\; d\mu(k) \\ 
        &= g\chi(c(g,x)) \int_{K} F(\alpha_g(x), k) \overline{\chi(\beta_{g^{-1}}(k))} \; d\mu(k)  \\ &= g\chi(c(g,x)) F_{g\chi}(\alpha_g(x)) 
    \end{align*}
Now if $g \in \Gamma_\chi$, we have that $F_\chi(x) = \chi(c(g,x))F_\chi(\alpha_g(x))$. Taking absolute values we have $|F_\chi(x)| = |F_\chi(\alpha_g(x))|$. Now since $\Gamma_\chi \actson X$ is ergodic, we have that $|F_\chi|$ is essentially constant. Suppose that $|F_\chi| = \delta_\chi 
\geq 0$. From the invariance equation above we also have that $\delta_{g\chi} = \delta_\chi$ for all $g \in \Gamma$. Thus $\delta$ is constant on each $\Gamma$-orbit of $\widehat{K}$. 

Now if $\delta_\chi > 0$ and $\chi$ has an infinite $\widehat{\beta}$-orbit, this implies that $\sum_{\chi \in \widehat{K}}|F_\chi|^2$ is infinite, which cannot happen. Thus $F_\chi = 0$ for every $\chi$ with an infinite orbit. If $\chi$ has a finite $\widehat{\beta}$-orbit, we get that $\phi \coloneqq \delta_\chi^{-1}F_\chi$ is a measurable $\T$ valued function. In fact for $h \in \Gamma_\chi$ we have: 
\begin{align*}
    \phi(x) = \delta_\chi^{-1} F_\chi(x) = \delta_\chi^{-1} \chi(c(h,x)) F_\chi(\alpha_h(x)) = \chi(c(h,x))\phi(\alpha_h(x)) 
\end{align*}
This implies that $c_\chi(h,x) = \chi(c(h,x)) = \phi(x)\phi(\alpha_h(x))^{-1}$ for a measurable function $\phi: X \rightarrow \T$ and hence $c_\chi$ is a coboundary. This is impossible by condition 2 and hence $\delta_\chi = 0$ for all $\chi \neq 1$. Thus $F$ is independent of the $K$ variable and $F(x,k) = f(x)$ for some $f \in L^\infty(X)$. Since $f$ stays $\alpha$-invariant, $f$ is essentially constant as required.

Now for $1 \implies 2$, suppose that $\chi \neq 1$ is a character with a finite $\widehat{\beta}$-orbit and a measurable function $b: X \rightarrow \T$ such that $c_\chi(g,x) = b(gx)^{-1}b(x)$ for all $g \in \Gamma_\chi$ and a.e. $x\in X$. For $k \in \Gamma$, define the function $B_{k\chi}$ on $X$ by $B_{k\chi}(x) =  \overline{(k\chi)(c(k,\alpha_{k^{-1}}(x)))}b(\alpha_{k^{-1}}(x))$. In fact this function only depends on the element $k\chi$ in the orbit $\Gamma\cdot \chi$, and not on the choice of $k$. Indeed if $k\chi = k'\chi$, then $k' = kh$ for some $h \in \Gamma_\chi$. Letting $y = \alpha_{k'^{-1}}(x)$, the cocycle identity gives: 
\begin{align*}
    c(kh, y) = c(k,\alpha_h(y))\beta_k(c(h,y)) = c(k, \alpha_{k^{-1}}(x))\beta_k(c(h,y))
\end{align*}
Since $k\chi \circ \beta_k = \chi$, we get: 
\begin{align*}
   k\chi(c(kh,y)) = k\chi(c(k, \alpha_{k^{-1}}(x)))\chi(c(h,y)) 
\end{align*}
Now since $c_\chi$ is a coboundary we get: 
\begin{align*}
    \overline{k\chi(c(kh,y))}b(y) &= \overline{k\chi(c(k, \alpha_{k^{-1}}(x))) \chi(c(h,y))}\chi(c(h,y))b(hy) \\ &= \overline{k\chi(c(k, \alpha_{k^{-1}}(x)))}b(hy)
\end{align*}
    Thus $B_{k\chi} = B_{k'\chi}$ and the function is well defined. Now define the function $F$ on $X \times K$ by:
\begin{align*}
    F(x,k) = \sum_{\eta \in \Gamma \cdot \chi} B_\eta(x)\eta(k)
\end{align*}
This is a bounded measurable function on $X \times K$ as $B_\eta$ is bounded measurable. We claim that $F$ is $\rho$-invariant. Notice that this is the same as showing $F_\chi(x) = g\chi(c(g,x))F_{g\chi}(\alpha_g(x))$ for all $g \in \Gamma$ for the Fourier coefficients as defined in the first part of the proof. Using the fact that $gk \chi \circ \beta_g = k\chi$ we get from the cocycle identity: 
\begin{align*}
    gk\chi(c(gk,\alpha_{k^{-1}}(x))) = gk\chi(c(g,x)) k\chi(c(k,\alpha_{k^{-1}}(x)))
\end{align*}
The above equation immediately gives $B_{k\chi}(x) = gk\chi(c(g,x)) B_{gk\chi}(\alpha_g(x))$. Hence $B_\eta(x) = g\eta(c(g,x))B_{g\eta}(\alpha_g(x))$ for all $\eta \in \Gamma \cdot \chi$. Now we calculate: 
\begin{align*}
    F_\chi(x) &= \int_{K} F(x,k)\overline{\chi(k)} \; d\mu(k) = \int_K \sum_{\eta \in \Gamma \cdot \chi} B_\eta(x) \eta(k)\overline{\chi(k)} \; d\mu(k) \\ &= \int_K B_\chi(x) \; d\mu(k) = \int_K g\chi(c(g,x))B_{g\chi}(\alpha_g(x)) \; d\mu(k) \\ &= g\chi(c(g,x))\int_K B_{g\chi}(\alpha_g(x)) \; d\mu(k) \\ &= g\chi(c(g,x))\int_K \sum_{\eta \in \Gamma \cdot \chi} B_\eta(\alpha_g(x))\eta(k)\overline{g\chi(k)} \; d\mu(k) \\ &= g\chi(c(g,x))\int_K F(\alpha_g(x),k)\overline{g\chi(k)} \; d\mu(k) = g\chi(c(g,x))F_{g\chi}(\alpha_g(x)) 
\end{align*}
This shows that $F$ is indeed $\rho$-invariant. Let us now observe that $F$ is non-constant. This is because the Fourier coefficient $F_\chi(x) = B_\chi(x) = \overline{c_\chi(e,x)}b(x) = b(x)$. Thus $|F_\chi(x)| = 1$ which cannot happen if $F$ is constant, thus contradicting ergodicity.  
\end{proof}

\begin{cor}
\label{Corr: diagonal product ergodic under stabilizer ergodicity}
    Suppose that $\alpha: \Gamma \actson X$ and $\beta: \Gamma \actson K$ are both ergodic. If the stabilizers $\Gamma_\chi$ act ergodically on $X$ for all $\chi \in \widehat{K}$ then the twisted skew product $\rho$ is ergodic. In particular when the twisted cocycle $c$ is trivial, then the diagonal product action is ergodic if $\Gamma_\chi \actson X$ is ergodic for all $\chi \in \widehat{K}$.    
\end{cor}
\begin{proof}
    This is clear from proposition $\ref{Prop: stabilizer ergodicity}$ as $\beta$ being ergodic implies all non-trivial $\widehat{\beta}$ orbits are infinite and hence condition 2 in Proposition \ref{Prop: stabilizer ergodicity} is vacuous. 
\end{proof}

From Proposition \ref{Prop: stabilizer ergodicity} we can recover the classical skew-product ergodicity result (see \cite[Corollary 5.4]{SchmidtLectureNotes}). 

\begin{cor}
    Suppose that $\alpha$ is ergodic and $\beta$ is trivial. Then $\rho$ is ergodic if and only if the essential range $E(c)$ of $c$ is $K$.
\end{cor}
\begin{proof}
    As before we shall denote the cocycle $c_\chi(g,x) = \chi(c(g,x))$, which is a genuine cocycle $\Gamma \times X \rightarrow \T$ now. Let us define the annihilator of the cocycle by $\Ann(c) = \{\chi \in \widehat{K} \; | \; c_\chi \text{ is a coboundary}\}$. We claim that $\Ann(c) = E(c)^{\perp}$, where $E(c)$ is the essential range of $c$. Notice that if the claim is true, then $\Ann(c) = \{1\}$ is exactly condition 2 in Proposition \ref{Prop: stabilizer ergodicity}. Clearly $\Ann(c) = \{1\}$ if and only if $E(c) = K$ and by the proposition, it is equivalent to ergodicity of the skew product $\rho$ as required. 

    Now we prove the claim. Suppose first that $\chi \in \Ann(c)$, so $c_\chi$ is a coboundary and there is a measurable function $b: X \rightarrow \T$ such that $\chi(c(g,x)) = b(x)b(gx)^{-1}$ for all $g \in \Gamma$ and a.e. $x \in X$. Suppose that there exists $k \in E(c) < K$ such that $\chi(k) \neq 1$. Pick $\epsilon > 0$ and let $U$ be an open neighbourhood of $k$ such that $|\chi(u) - 1| > \epsilon$ for all $u \in U$. Pick a point $t \in \T$ and choose a measurable subset $E \subset X$ with $0 < \mu(E) < \infty$ such that $|b(x) - t| < \epsilon/4$ for all $x \in E$. Since $k \in E(c)$, there is a $g \in \Gamma$ such that: 
    \begin{align*}
        \mu(E \cap g^{-1}E \cap \{x \; | \; c(g,x) \in U\}) > 0 
    \end{align*}
    So there exists a positive measure subset $E_0 < E$ such that for all $x \in E_0$ we have that $gx \in E$ and thus: 
    \begin{align*}
       &|b(x)b(gx)^{-1} - 1| = |b(x)b(gx)^{-1} - tt^{-1}| \\ = &|b(x)b(gx)^{-1} - b(x) t^{-1} + b(x)t^{-1} - tt^{-1}| \leq \epsilon/2
    \end{align*}
    However since $b(x)b(gx)^{-1} = \chi(c(g,x))$ and $c(g,x) \in U$, in fact we get $|\chi(c(g,x)) - 1|> \epsilon$, which gives a contradiction. Thus $\chi(k) = 1$ for all $k \in E(c)$ and $\chi \in E(c)^{\perp}$.  

    Conversely, since $E(c) = K$ let $c'$ be a cohomologous cocycle to $c$ such that $c'$ takes values in $E(c)$. Since $\chi(k) = 1$ for all $k \in E(c)$, therefore $c'_\chi = \chi \circ c'  = 1$ and hence $c_\chi = \chi \circ c$ is a coboundary as required.  
\end{proof}

\section{Actions on abelian groups of Lie type}
\label{Sec: The general case: double ergodicity}

In this section we shall consider actions $\Gamma \actson N$ on general locally compact abelian groups of Lie type and try to determine when $L(N \rtimes \Gamma)$ is a factor. By Corollary \ref{Cor: implications of factoriality}, $N$ has to be of the form $\R^n \times D$ for a discrete countable group $D$. Such actions have a specific form as we describe below: 

\begin{rem}
    \label{Rem: form of actions}
Consider an action  $\Gamma \actson \R^n \times K$ by measure preserving continuous automorphisms for a compact group $K$. For any $g \in \Gamma$, we have that $gK = K$ as $K$ is compact and $\R^n$ does not have any compact subgroups. As a consequence, we have that the action $\Psi: \Gamma \actson \R^n \times K$ is given by $\Psi_g(x,y) = (\alpha_g(x), c(g,x)\beta_g(y))$ for actions $\alpha: \Gamma \actson \R^n$, $\beta: \Gamma \actson K$ by continuous automorphisms and for a Borel map $c: \Gamma \times \R^n \rightarrow K$. The map $c$ is not a cocycle for the action $\Gamma \times \R^n$, but rather a `twisted cocycle', i.e., it satisfies: 
\begin{align*}
    c(h, gx) \beta_h(c(g,x)) = c(hg,x)
\end{align*}
We will call such a map a \textit{$\beta$-cocycle} for the action $\alpha: \Gamma \actson \R^n$. For convenience of notation, in this case we shall denote the action $\Psi$ by $(\alpha,\beta,c)$. If $c$ is trivial, then $\Psi: \Gamma \actson \R^n \times K$ is a diagonal product of $\alpha$ and $\beta$ and we shall say that the action $\Psi$ \textit{splits}. 

If $D = \widehat{K}$, the dual action $\Phi \coloneqq \widehat{\Psi}$ similarly decomposes into actions $\delta \coloneqq \widehat{\alpha}: \Gamma \actson \R^n$, $\eta \coloneqq \widehat{\beta}: \Gamma \actson D$ and a $\delta$-cocycle $ \omega \coloneqq \widehat{c}: \Gamma \times D \rightarrow \R^n$ for $\eta$, such that $\Phi$ is the twisted skew product given by:
\begin{align*}
    \Phi_g(x,\chi) = (\delta_g(x)\omega(g,\chi), \eta_g(\chi))
\end{align*}
In this case we shall call $\omega$ the \textit{dual cocycle} of $c$. One checks that the dual cocycle $\omega(g,\chi)$ is the unique element of $\widehat{\R^n}$ satisfying: 
\begin{align*}
    e^{2\pi i \langle \omega(g,\chi),x \rangle} = \chi(c(g^{-1},x))
\end{align*}
\end{rem}

Letting $\Phi^0$ be the diagonal action $\delta \times \eta$, we have two distinct cases to deal with, when the dual of $\Phi^0$ is already ergodic, we can apply Proposition \ref{Prop: diagonal action ergodic implies skew product ergodic} to show that the dual of $\Phi$ is also ergodic. We deal with this case in the first subsection below. In the second situation when the dual of $\Phi^0$ is not necessarily ergodic anymore, we can apply Proposition \ref{Prop: stabilizer ergodicity} and give some partial conditions for factoriality, that we record in the second subsection. 

\subsection{Ergodicity of the diagonal action}

\begin{thm}
\label{Thm: main theorem on R^n times K when diagonal action ergodic}
    Let $\Gamma$ be a discrete countable group and $\Psi = (\alpha,\beta,c)$ be a Haar-preserving action of $\Gamma$ on $\R^n \times K$ for a compact abelian group $K$ by continuous group automorphisms. Let $\Psi_0$ be the diagonal product action $\alpha \times \beta$. Suppose one of the following conditions is satisfied: 
    \begin{enumerate}
        \item The action $\alpha: \Gamma \actson \R^n$ is faithful and doubly ergodic and the action $\beta: \Gamma \actson K$ is ergodic.
        \item  The action $\alpha: \Gamma \actson \R^n$ is faithful and ergodic and the action $\beta: \Gamma \actson K$ is mixing.
    \end{enumerate}
    Then $\Psi_0$ and $\Psi$ are both ergodic and the the crossed product $L^{\infty}(\R^n \times K) \rtimes_{\Psi} \Gamma$ is a factor. Ergodicity of $\alpha$ is also necessary for factoriality.
\end{thm}

\begin{proof}
      Since $\alpha$ is faithful, it is essentially free as in Theorem \ref{Thm: characterization of factoriality for actions on R^n}. As a consequence $\Psi$ is essentially free. Suppose condition 1 holds, since $\alpha$ is doubly ergodic, it is weakly mixing by \cite[Theorem 1.1]{GLASNER_WEISS_2016}. Since $K$ is compact, $\beta$ is a pmp ergodic action and hence the diagonal action $\Psi_0$ is ergodic. Now suppose condition 2 holds, then $\alpha$ is properly ergodic as $\Gamma$ is countable, and by Theorem \ref{Thm: Schmidt-Walters theorem}, $\Psi_0$ is ergodic. Thus $\Psi$ is ergodic by Proposition \ref{Prop: diagonal action ergodic implies skew product ergodic} as required.  
\end{proof}

Now we can restate the above in terms of factoriality of semidirect products as follows. Let $N$ be a locally compact abelian group, as before we shall denote the connected component of the identity of $N$ by $N^{\circ}$. We can write $N \cong N^\circ \times N/N^\circ$. As in Remark \ref{Rem: form of actions}, such an action is given by a triplet $(\delta,\eta,\omega)$ where $\delta$ and $\eta$ are actions on $N^\circ$ and $N/N^\circ$ by continuous automorphisms and $\omega: \Gamma \times N/N^\circ \rightarrow N^\circ$ is a $\delta$-cocycle for $\eta$.

\begin{proof}[Proof of Theorem \ref{Thm: main theorem on R^n times D, diagonal action ergodic}]
    By Theorem \ref{Thm: structure of abelian locally compact groups}, we have that $N \cong \R^m \times \T^n \times D$ for a discrete abelian group $D$. Absence of compact connected subgroups implies that $N \cong \R^m \times D$. Consider the dual action $\Psi = \widehat{\Phi}: \Gamma \actson \R^m \times \widehat{D}$. Let $\alpha = \widehat{\delta} : \Gamma \actson \R^m$ and $\beta = \widehat{\eta}: \Gamma \actson \widehat{D}$ denote the components of the dual action. Suppose condition 1 holds, then $\alpha$ is faithful and doubly ergodic. By Proposition \ref{Prop: ergodicity iff infinite orbits}, we have that $\beta$ is ergodic. Then by Theorem \ref{Thm: main theorem on R^n times K when diagonal action ergodic}, the crossed product $L^{\infty}(\R^m \times \widehat{D}) \rtimes_{\Psi} \Gamma$ is a factor. Hence $L(N \rtimes \Gamma)$ is a factor by Theorem \ref{Thm: von neumann algebra of sd product is crossed product of dual action}. Now if condition 2 holds, by Theorem \ref{Theorem: mixing iff no ergodic subgroups}, the dual action $\Gamma \actson \widehat{D}$ is pmp and mixing. Since $\Gamma \actson \R^n$ is obviously properly ergodic, the diagonal action stays ergodic by Theorem \ref{Thm: Schmidt-Walters theorem} and then Proposition \ref{Prop: diagonal action ergodic implies skew product ergodic} gives factoriality.
    
    Conversely, if $L(N\rtimes \Gamma)$ is a factor, then by Corollary \ref{Cor: implications of factoriality}, we get that $N^\circ$ has no non-trivial compact subgroups. Dual ergodicity of $\delta$ follows from the fact the $\alpha$ has to be ergodic if $\Psi$ is ergodic. 
\end{proof}

Theorem \ref{Thm: main theorem on R^n times D, diagonal action ergodic} applies to many natural classes of examples. In particular, we prove Corollary \ref{Corr: splitting actions factoriality for lattices} here: 

\begin{proof}[Proof of Corollary \ref{Corr: splitting actions factoriality for lattices}]
    By \cite[Lemma 5.6]{PopaVaes11}, if $\Gamma$ is a lattice in $\SL(n,\R)$ for $n \geq 3$, then $\Gamma \actson \R^n$ is doubly ergodic. This implies in particular that $\SL(n,\R) \actson \R^n$ is doubly ergodic, which in turn implies that if $\Gamma$ is a dense subgroup then $\Gamma \actson \R^n$ is also doubly ergodic. Clearly such actions are also dually doubly ergodic. The result now follows from Theorem \ref{Thm: main theorem on R^n times D, diagonal action ergodic}.  
\end{proof}

\begin{exmp}
\label{Example: splitting actions on R^m times Z^n}
    In Corollary \ref{Corr: splitting actions factoriality for lattices}, one can consider $D = \Z^n$ and $\Gamma$ any finite index subgroup of $\SL(n,\Z)$. Then for the linear action $\Gamma \actson \R^n \times \Z^n$, the von Neumann algebra of the semidirect product is a factor.  
\end{exmp}

\begin{rem}
\label{Rem: torsion case}
    In the case where $D = N/N^\circ$ is a torsion abelian group, the twisted cocycle always vanishes and any action on $N \times D$ is a diagonal product action. Indeed, for the dual cocycle and each $g \in \Gamma$, the map $n \mapsto \widehat{c}(g,n)$ is a group homomorphism $D \rightarrow \R^n$. Since $D$ is a torsion group, this group homomorphism is trivial for all $g \in \Gamma$. Consequently, the action automatically splits into $\delta$ and $\eta$. 
\end{rem}

\begin{exmp}
    An example of a countable abelian torsion group is the Prüfer $p$-group $\Z_{p^{\infty}}$ for a prime $p$. Recall that $\Z_{p^\infty}$ is the direct limit of cyclic groups of order $p^n$ as $n \rightarrow \infty$. Every countable abelian group $D$ is a direct sum of its divisible and reduced parts $D_d$ and $D_r$  respectively. Recall that a group is called divisible if for all $d \in D$ and positive integer $n$, there is an element $a \in D$ such that $a^n = d$. It is called reduced if the only divisible subgroup is trivial. It turns out that every countable divisible abelian group is a direct sum of copies of $\Q$ and $\Z_{p^{\infty}}$. Hence direct sums of $\Z_{p^\infty}$ are the only divisible torsion abelian groups. There are examples of reduced torsion abelian groups as well: any direct sum of finite cyclic groups. Recall that in Proposition \ref{Prop: examples of mixing actions by Bernoulli shifts} we get examples of actions on such groups such that the dual action is ergodic and even mixing.
    
    There is an elaborate theory of classifying such reduced abelian groups in terms of their so-called Ulm invariants. In fact a result of Kulikov \cite{Kulikov45} says that any torsion abelian group is an extension of a direct sum of finite cyclic groups by a divisible group. We point the reader to \cite{Hill_Survey} for a nice survey of the literature on classification of countable abelian groups.  
\end{exmp}

\begin{rem}
    We remark here that the countable abelian group is torsion if and only if its compact Pontryagin dual is totally disconnected. Notice that as in Remark \ref{Rem: torsion case}, any action of $\Gamma$ on $\R^n \times K$ for a compact totally disconnected group is actually a diagonal action. In particular the cocycle $\Gamma \times \R^n \rightarrow \Affn(K)$ is trivial.
\end{rem}

Recall from Example \ref{example: ergodic actions on R^n} that our examples of linear actions on $\R^n$ (for $n\geq 3$) that are ergodic are also doubly ergodic (for example, lattices and dense subgroups of $\SL(n,\R)$ for $n \geq 3$). Therefore, to apply Condition 2 in \ref{Thm: main theorem on R^n times D, diagonal action ergodic} we need examples of actions on $\R^n$ that are ergodic but not doubly ergodic. One class of examples come from lattices in $\SL(2,\R)$ since their linear actions are not doubly ergodic (see \cite[Lemma 5.6]{PopaVaes11}). 

For the next proposition we note that there exists subgroups of $\SL(2,\Z)$ which are isomorphic to free groups and still have finite index in $\SL(2,\Z)$. This is a strictly 2-dimensional phenomenon, as it can be checked that for $n \geq 3$, a free subgroup of $\SL(n,\Z)$ cannot be finite index. For example, consider the so-called principal congruence subgroups of $\SL(2,\Z)$. For $n \geq 2$, the \textit{principal congruence subgroup of level $n$} denoted by $\Gamma(n) < \SL(2,\Z)$ is the kernel of the surjective group homomorphism $ \rightarrow \SL(2, \Z/n\Z)$. It can also be described as: 
\begin{align*}
    \Gamma(n) = \left\{
\begin{pmatrix}
a & b \\
c & d
\end{pmatrix} \in \mathrm{SL}(2, \mathbb{Z}) \;\middle|\;
a \equiv d \equiv 1 \pmod{n}, \quad b \equiv c \equiv 0 \pmod{n}
\right\}
\end{align*}

It is easy to see that principal congruence subgroups are finite index subgroups of $\SL(2,\Z)$, and hence they are lattices in $\SL(2,\R)$. Hence all principal congruence subgroups act ergodically on $\R^2$. It turns out that for $n \geq 3$, the group $\Gamma(n)$ is isomorphic to a free group.  

Another example of a free subgroup of finite index in $\SL(2,\Z)$ is the Sanov subgroup, first defined in \cite{Sanov}). This is the subgroup generated by the elements: 
\begin{align*}
    \begin{pmatrix}
        1 & 2 \\ 0 & 1 
    \end{pmatrix} \text{   and   } \begin{pmatrix}
        1 & 0 \\ 2 & 1
    \end{pmatrix}
\end{align*}
A classic application of the ping-pong lemma shows that the group generated by these two elements is $\F_2$. It can be checked by writing out the standard form of elements that the Sanov subgroup has index 2 in $\Gamma(2)$, and hence is in particular a finite index subgroup of $\SL(2,\Z)$. Thus it is a lattice and acts ergodically on $\R^2$. 

\begin{prop}
\label{Prop: actions of subgroups of SL(2,Z) that are mixing and ergodic}
    Let $\Gamma$ be a finite-index torsion-free subgroup of $\SL(2,\Z)$ (for example if $\Gamma$ is isomorphic to $\F_n$ for some $n > 1$). Let $\Lambda$ be any non-trivial group and consider the action $\eta: \Gamma \actson \bigoplus_{\Gamma} \Lambda $ by Bernoulli shifts. Let $\omega: \Gamma \times \bigoplus_{\Gamma} \Lambda  \rightarrow \R^2$ be any multiplicative twisted cocycle and let $\Psi = (\delta,\eta,\omega)$ be the twisted skew product action. Then $L((\R^2 \times \bigoplus_{\Gamma}\Lambda) \rtimes_\Psi \Gamma)$ is a factor. In particular when the cocycle is trivial, the diagonal product action gives a factor.  
\end{prop}

\begin{proof}
    Let $H = \bigoplus_{\Gamma} \Lambda$ and notice that $\Gamma \actson H \backslash \{\id\}$ is free by Proposition \ref{Prop: examples of mixing actions by Bernoulli shifts}. Since $\Gamma$ is a finite index subgroup of $\SL(2,\Z)$, we know that $\Gamma$ is a lattice and hence $\Gamma \actson \R^2$ is dually ergodic. The result now follows from Theorem \ref{Thm: main theorem on R^n times D, diagonal action ergodic}.
\end{proof}

We remark here that there are ways to construct examples of ergodic actions $\Gamma \actson \R^n$ that are not doubly ergodic when $n > 2$ as well. For example let $\Gamma < \SL(n,\R)$ be a lattice. By \cite[Lemma 5.6]{PopaVaes11}, the $k$-fold diagonal action $\Gamma \actson (\R^n)^{(k)}$ is ergodic if and only if $k \leq n-1$. Pick $k$ such that $k \leq n-1 < 2k$. Consider the diagonal inclusion of $\SL(n,\R)$ in $\SL(nk,\R)$ and
let $\Gamma^k < \SL(nk,\R)$ denote the image of $\Gamma$. Then the linear action $\Gamma^k \actson \R^{nk}$ is ergodic but $\Gamma^k \actson \R^{nk} \times \R^{nk}$ is not ergodic because this is precisely the $2k$-fold diagonal action $\Gamma \actson (\R^n)^{(2k)}$. By a well-known result of Selberg, every lattice in $\SL(n,\R)$ has torsion-free finite index subgroups, so there is a plethora of lattices in $\SL(n,\R)$ that are torsion-free, and hence one can construct examples similar to Proposition \ref{Prop: actions of subgroups of SL(2,Z) that are mixing and ergodic} in higher dimensions as well.

\subsection{The diagonal action is not necessarily ergodic}

\begin{defn}
\label{Def: Cob_fin}
    Let $\Phi: \Gamma \actson \R^n \times D \cong \widehat{\R^n} \times \widehat{K}$ be given by $(\delta,\eta,\omega)$ as in Remark \ref{Rem: form of actions}. For each $\chi \in D$, we shall denote the $\eta$-stabilizer of $\chi$ by $\Gamma_\chi$. We define the genuine cocycle $\omega_\chi: \Gamma_\chi \times \R^n \rightarrow \T$ by:
    \begin{align*}
        \omega_\chi(h,x) = \exp(2\pi i \langle \omega(h^{-1},\chi), x \rangle)
    \end{align*}
    We denote by $\Cob_{\fin}(\omega)$ the set of elements $\chi \in D$ such that $\chi$ has a finite $\eta$-orbit and $\omega_\chi$ is a coboundary. 
\end{defn}

\begin{lem}
\label{Lemma: little lemma about dual cocycles}
    Let $\Phi = (\delta, \eta, \omega): \Gamma \actson \R^n \times D$ be the dual of $\Psi = (\alpha, \beta, c): \Gamma \actson \R^n \times K$. Suppose that $\delta|_{\Gamma_\chi}: \Gamma_\chi \actson \R^n$ is dually ergodic for all $\chi \in D$. Then $\Psi$ is ergodic if and only if $\Cob_{\fin}(\omega) = \{e\}$. 
\end{lem}
\begin{proof}
    By Proposition \ref{Prop: stabilizer ergodicity}, it is enough to show that if $\chi$ has a finite $\eta$-orbit then $c_\chi$ is a coboundary if and only if $\omega_\chi$ is a coboundary. But this is clear from definition because $\omega_\chi(h,x) = \chi(c(h,x))$.  
\end{proof}

Now we have all the ingredients to prove our final main theorem. 

\begin{proof}[Proof of Theorem \ref{Thm: diagonal action nor ergodic}]
    Since there are no compact connected subgroups we can assume $N \times N/N^\circ = \R^n \times D$ for a discrete abelian group $D$. Suppose the dual action is $\Psi = (\alpha,\beta,c)$ on $\R^n \times K$. By condition 1, $\alpha$ is faithful and hence essentially free, thus $\Psi$ is essentially free. By condition 3, $\Cob_{\fin}(\omega) = \{e\}$ and hence by Lemma \ref{Lemma: little lemma about dual cocycles}, $\Psi$ is ergodic. Thus $L^\infty(\R^n \times K) \rtimes_\Psi \Gamma = L(N \rtimes_\Phi \Gamma)$ is a factor. 
\end{proof}

\begin{exmp}
\label{Exmp: individual ergodicity not enough}
    Let us give an example when $\alpha$ is not doubly ergodic and $\beta$ is not mixing (so we are outside the framework of Theorem \ref{Thm: main theorem on R^n times K when diagonal action ergodic}) and still the diagonal product is ergodic by Corollary \ref{Corr: diagonal product ergodic under stabilizer ergodicity} and hence the dual action gives a factor by Theorem \ref{Thm: diagonal action nor ergodic}. 
    
    Let $\Gamma$ be a a dense copy of $\F_2 = \langle a,b \rangle$ inside $\SL(2,\R)$, which exists by the man result in \cite{BreuillardGelander}. Let $q$ be the quotient from $\Gamma$ to $ \langle b \rangle = \Z$. Let $\Lambda = \ker(q)$, since $\Lambda$ is a normal subgroup of $\Gamma$, $\overline{\Lambda}$ is normalised by $\Gamma$ (because conjugation is a homeomorphism). Since $\overline{\Gamma} = \SL(2,\R)$, we have that $\overline{\Lambda}$ is a closed normal subgroup of $\SL(2,\R)$. Moreover $\overline{\Lambda}$ is not contained in the center because $a \in \Lambda$. Thus $\overline{\Lambda} = \SL(2,\R)$ and hence $\Lambda$ is dense. 
    
    Notice that the linear action $\alpha: \Gamma \actson \R^2$ is ergodic since $\Gamma$ is dense but not doubly ergodic because $\SL(2,\R) \actson \R^2 \times \R^2$ is not ergodic. Consider now the following matrix and action $\beta: \Gamma \actson \T^2$: 
    \begin{align*}
        B = \begin{pmatrix}
            2 & 1 \\ 1 & 1
        \end{pmatrix} \in \SL(2,\Z),  \;\;\; \beta_g(t) = B^{q(g)}(t) \text{ for all } t \in \T^2, g \in \Gamma 
    \end{align*}
    Notice that the dual action $\Gamma \actson \Z^2$ is given by $g \cdot n = (B^{-q(g)})^T n$ and since the image of $q$ is $\Z$, the $\Gamma$-orbit of $n \in \Z^2$ is the orbit under the linear action by $\langle B^T \rangle$. The matrix $B$ has no root of identity as eigenvalues hence the non-trivial orbits of this action are all infinite. Thus $\beta: \Gamma \actson \T^2$ is ergodic. 

    Since the kernel of $q$ is infinite, $B^q(g)$ is trivial for infinitely many $g \in \Gamma$ and hence the stabilizers of a.e. $n \in \Z^2$ are infinite. Therefore the action $\Gamma \actson \T^2$ is not mixing. Let $\Gamma_n$ be the stabilizer of the point $n \in \Z^2$. Notice that $\Lambda < \Gamma_n$ for all $n \in \Z^2$  and since $\Lambda$ is dense in $\SL(2,\R)$, the action $\Lambda \actson \R^2$ is already ergodic, and hence $\Gamma_n \actson \R^2$ is ergodic for all $n \in \Z^2$. Thus by Proposition \ref{Prop: stabilizer ergodicity}, the diagonal action is ergodic. Of course, one can add non-trivial cocycles to the constructions above to get more examples of groups satisfying the conditions of Theorem \ref{Thm: diagonal action nor ergodic}. 
\end{exmp}

\begin{rem}
\label{rem: type of the factor}
    Note that in our main Theorems \ref{Thm: main theorem on R^n times D, diagonal action ergodic} and \ref{Thm: diagonal action nor ergodic}, if $N$ is not discrete, then the factor $L(N \rtimes \Gamma)$ is of type II$_\infty$. This is because the dual action $\Gamma \actson \widehat{N}$ is an infinite measure preserving action on a diffuse standard Borel space. A lot of our results have obvious generalizations to the case when the action is not measure preserving anymore (for example the linear action $\GL(n,\Z) \actson \R^n$). In that case the factors that we obtain are often of type III. We also remark that the class of examples we obtain in this article cover injective and non-injective II$_\infty$ factors. By the results of \cite{PopaVaes11}, when $n \geq 3$, actions of lattices in $\SL(n,\R)$ on $\R^n$ have property (T) and hence the crossed products are not injective. However for example, as in \cite{Aubert1981}, $\SL(2,\Z) \actson \R^2$ is amenable and the crossed product is the injective II$_\infty$ factor.    
\end{rem}

\printbibliography

\end{document}